\documentclass[reqno]{amsart}
\usepackage[foot]{amsaddr}
\makeatletter
\def\amsartsection{\@startsection{section}{1}%
	\z@{2.5\linespacing\@plus\linespacing}{.5\linespacing}%
	{\normalfont\scshape\centering}}

\def\section{\vspace*{-34pt}\vspace{\baselineskip}\let\section\amsartsection\section}
\makeatother
\usepackage[indentafter]{titlesec}
\usepackage{bbm}
\titleformat{name=\section}{}{\thetitle.}{0.8em}{\centering\scshape}
\titleformat{name=\subsection}[runin]{}{\thetitle.}{0.5em}{\bfseries}[.]
\titleformat{name=\subsubsection}[runin]{}{\thetitle.}{0.5em}{\itshape}[.]
\titleformat{name=\paragraph,numberless}[runin]{}{}{0em}{}[.]
\titlespacing{\paragraph}{0em}{0em}{0.5em}
\titleformat{name=\subparagraph,numberless}[runin]{}{}{0em}{}[.]
\titlespacing{\subparagraph}{0em}{0em}{0.5em}
\usepackage[utf8]{inputenc}
\usepackage{amsmath}
\usepackage{graphicx}
\usepackage{epstopdf}
\usepackage{physics}
\usepackage{setspace}
\usepackage[noadjust]{cite}

\usepackage{lipsum}
\usepackage{array}
\numberwithin{equation}{section}
\usepackage[svgnames]{xcolor}
\definecolor{dblue}{rgb}{0.0, 0.0, 0.55}
\usepackage[labelfont={color=dblue,bf}]{caption}
\usepackage[english]{babel}
\usepackage{amssymb}
\usepackage[numbers,sort&compress,square]{natbib}
\usepackage{enumitem}
\usepackage{titlesec}
\usepackage{lipsum}
\usepackage{float}

\usepackage{algorithm,algcompatible,amsmath}

\algnewcommand\INPUT{\item[\textbf{Input:}]}%
\algnewcommand\OUTPUT{\item[\textbf{Output:}]}%
\usepackage{tikz}
\usepackage{eso-pic}
\usepackage{pdfpages}
\usepackage{afterpage}
\usepackage{setspace}
\titlespacing{}{0pt}{0pt}{0pt} 
\usepackage{xcolor}
\usepackage{amsthm}
\usepackage[left=1in,right=1in,bottom=1in,top=1in]{geometry}
\theoremstyle{plain}
\newtheorem{thm}{Theorem}[section]
\newtheorem{defi}{Definition}[section]

\newtheorem{lem}{Lemma}[section]
\newtheorem{rem}{Remark}[section]

\usepackage{float}
\usepackage{multicol}
\usepackage{hyperref}
\hypersetup{colorlinks,
	linkcolor={dblue},
	citecolor={red},
	urlcolor={violet}}
\usepackage[figurename=Fig.]{caption}  
\title[Well-posedness results for a new class of stochastic SIR-type models]{Well-posedness results for a new class of stochastic spatio-temporal SIR-type models driven by proportional pure-jump Lévy noise}

\author[M. MEHDAOUI]{Mohamed Mehdaoui$^{\dagger}$}
\address{$\dagger$MAIS Laboratory, MAMCS Group, FST Errachidia, Moulay Ismail University of Meknes, P.O. Box 509, Boutalamine 52000, Errachidia, Morocco.}
\curraddr{}
\email{m.mehdaoui@edu.umi.ac.ma}



\date{}
\date{}
\begin{document}
	\setstretch{1}
	\maketitle 
\begin{abstract}
	This paper provides a first attempt to incorporate the massive discontinuous changes in the spatio-temporal dynamics of epidemics. Namely, we propose an extended class of epidemic models, governed by coupled stochastic semilinear partial differential equations, driven by pure-jump Lévy noise. Based on the considered type of incidence functions, by virtue of semi-group theory, a truncation technique and Banach fixed point theorem, we prove the existence and pathwise uniqueness of mild solutions, depending continuously on the initial datum. Moreover, by means of a regularization technique, based on the resolvent operator, we acquire that mild solutions can be approximated by a suitable converging sequence of strong solutions. With this result at hand, for positive initial states, we derive the almost-sure positiveness of the obtained solutions. Finally, we present the outcome of several numerical simulations, in order to exhibit the effect of the considered type of stochastic noise, in comparison to Gaussian noise, which has been used in the previous literature. Our established results lay the ground-work for investigating other problems associated with the new proposed class of  epidemic models, such as asymptotic behavior analyses, optimal control as well as identification problems, which primarily rely on the existence and uniqueness of biologically feasible solutions.
\end{abstract}
		\section{Introduction and motivation}
		\label{intro}
		As a scientific tool, mathematical modeling of infectious diseases has attracted the attention of several researchers throughout the decades. Primarily, due to its capacity of using several mathematical objects, in order to convert the biological real problem into a mathematical one, which once solved, through a deductive reasoning, provides valuable information that can be interpreted and analyzed, leading to the better understanding and identification of different factors altering the evolution of a given disease within a studied population. In this framework, to describe the propagation of the Plague, compartmental modeling was introduced in 1927 by Kermack and McKendrick \cite{kermack1927contribution}. The primal idea of their approach was to divide the population under study into three compartments, denoted as $S(t)$, $I(t)$ and $R(t)$, which stand, at a given time $t>0$, for the individuals susceptible to the disease, those infected with it, and last, those who recover or die from the disease and are thus removed from the studied population. Then, by introducing two positive parameters $\beta$ and $\gamma,$ which stand for the disease transmission and recovery rates, respectively, they proposed a model which is governed by the following system of coupled nonlinear ordinary differential equations:
		\begin{equation}\label{eq:sirode}
			\begin{cases}
				\begin{split}
					&\dfrac{dS(t)}{dt}=-\beta S(t) I(t),\\[1.5ex]
					&\dfrac{dI(t)}{dt}=\beta S(t) I(t)-\gamma I(t),\\[1.5ex]
					&\dfrac{dR(t)}{dt}=\gamma I(t).
				\end{split}
			\end{cases}
		\end{equation}
		
		Compartmental epidemic models which are governed by coupled nonlinear ordinary differential equations, follow a deterministic approach which assumes that the susceptible, infected and recovered populations' densities, are completely and surely determined, provided that the parameters and initial data are known. However, as it has been demonstrated by May \cite{may2019stability} and Mao \emph{et al.}  \cite{mao2002environmental}, the random fluctuations play a major role in the dynamics of ecological and epidemiological systems. Taking this fact into account, Jiang \emph{et al.} \cite{jiang2011asymptotic} extended the standard deterministic SIR model by taking the random fluctuations into account and by introducing three additional positive parameters $\Lambda, \mu,$ and $\epsilon$, standing for the birth rate, natural death rate and the rate of death caused by the disease. Namely, the authors considered the following stochastic model:
		\begin{equation}\label{eq:sirsde}
			\begin{cases}
				\begin{split}
					&{dS(t)}=(\Lambda-\beta S(t) I(t)-\mu S(t))dt+\sigma_1 S(t)dB_1(t),\\[1.6ex]
					&{dI(t)}=(\beta S(t) I(t)-(\gamma+\mu+\epsilon)I(t))dt+\sigma_2 I(t) dB_2(t),\\[1.6ex]
					&{dR(t)}=(\gamma I(t)-\mu R(t))dt+\sigma_3 R(t) dB_3(t),
				\end{split}
			\end{cases}
		\end{equation}
		where $(B_i(t))_{i \in \{1,2,3\}}$ are independent Brownian motions with corresponding noise intensities $(\sigma_i^2)_{i \in \{1,2,3\}}$.

		The sudden severe environmental changes, which are exhibited for instance by volcanoes, hurricanes and tornadoes, may cause the discontinuity of solutions. In this case, Gaussian noise is no longer suitable. To overcome this limitation, Zhang and Wang \cite{zhang2013stochastic} extended Model \eqref{eq:sirsde} by considering, in addition to multiplicative Gaussian noise, a Lévy-type jump process, as follows:
		\begin{equation}\label{eq:sirsdelevy}
			\begin{cases}
				\begin{split}
					&\displaystyle{dS(t)}=(\Lambda-\beta S(t) I(t)-\mu S(t))dt+\sigma_1 S(t)dB_1(t)+\int_{\mathbb{Z}} C_1(z) S(t-) \tilde{N}(dt,dz),\\
					&{dI(t)}=(\beta S(t) I(t)-(\gamma+\mu+\epsilon)I(t))dt+\sigma_2 I(t) dB_2(t)+\int_{\mathbb{Z}} C_2(z) I(t-) \tilde{N}(dt,dz),\\
					&{dR(t)}=(\gamma I(t)-\mu R(t))dt+\sigma_3 R(t) dB_3(t)+\int_{\mathbb{Z}} C_3(z) R(t-) \tilde{N}(dt,dz),
				\end{split}
			\end{cases}
		\end{equation}
		where $S(t-)$, $I(t-)$ and $R(t-)$ stand for the right limits of $S(t),$ $I(t)$ and $R(t),$ respectively. $\tilde{N}(dt,dz)$ is the compensated Poisson process, with characteristic measure $\nu$ defined on a measurable subset $\mathbb{Z}\subset [0,+\infty),$ such that $\nu(\mathbb{Z})<+ \infty$. Additionally, $C_i(z)>-1$ for $i \in \{1,2,3\}.$

		For infectious diseases, the movements of individuals within the population contribute to a great extent. Hence, time-dependent epidemic models are no longer qualified due to omitting the spatial factor. Taking this fact into account, Webb \cite{webb1981reaction} considered in a one spatial dimension, the following adapted SIR model: 
		\begin{equation}\label{eq:sirpde}
			\begin{cases}
				\begin{split}
					&\partial_t S(x,t)-\partial_{xx} S(x,t)=-\beta S(x,t) I(x,t), &\quad -L<x<L, \quad t>0,\\[1.5ex]
					&\partial_t I(x,t)-\partial_{xx} I(x,t)=\beta S(x,t) I(x,t)-\gamma I(x,t), &\quad -L<x<L, \quad t>0,\\[1.5ex]
					&\partial_t R(x,t)-\partial_{xx} R(x,t)=\gamma I(x,t), &\quad -L<x<L, \quad t>0,
				\end{split}
			\end{cases}
		\end{equation}
		equipped with the positive initial conditions: 
		\begin{equation}\label{bb1}
			S(x,0)\geq 0, \quad I(x,0)\geq 0 \quad \text{and}\quad 	R(x,0)\geq 0, \quad -L<x<L,
		\end{equation}
		and the following homogeneous Neumann boundary conditions:
		\begin{equation}\label{eq:ii1}
			\partial_x S(\pm L,t)= \partial_x I(\pm L,t)=\partial_x R(\pm L,t)=0, \quad t>0,
		\end{equation}
		where $L>0$, $\partial_x$ and $\partial_t$ stand for the first-order partial derivatives with respect to the spatial and time variables, respectively, while $\partial_{xx}$ stands for the second-order derivative with respect to the spatial variable.

		In order to incorporate the spatial factor in stochastic epidemic models, Nguyen \emph{et al.} \cite{nguyen2019stochastic} considered the following SIS model:
		\begin{equation}\label{eq:spdesir}
			\begin{cases}
				\begin{split}
					&\begin{split}
						d S(t, x)&=\big[k_1 \Delta S(t, x)+\Lambda(x)-\mu_1(x) S(t, x)-\frac{\alpha(x) S(t, x) I(t, x)}{S(t, x)+I(t, x)}\\[1.5ex]
						&+\gamma(x) I(t,x)\big] d t +S(t, x) d W_1(t, x), 
					\end{split} &\; \text {in}\; \mathbb{R}^{+} \times \mathcal{O},\\[1.5ex]
					& \begin{split} 
						dI(t, x)&=\big[k_2 \Delta I(t, x)-\mu_2(x) I(t, x)+\frac{\alpha(x) S(t, x) I(t, x)}{S(t, x)+I(t, x)}\\
						&-\gamma(x) I(t,x)\big] dt+I(t, x) d W_2(t, x), 
					\end{split} &\; \text {in}\; \mathbb{R}^{+} \times \mathcal{O},
				\end{split}
			\end{cases}
		\end{equation}
		equipped with the following homogeneous Neumann boundary conditions: 
		\begin{equation}\label{eq:bb2}
			\partial_\nu S(t, x)=\partial_\nu I(t, x)=0, \quad (t,x)\in \mathbb{R}^{+} \times \partial \mathcal{O},
		\end{equation}
		and the initial conditions:
		\begin{equation}\label{eq:ii2}
			S(0, x)\geq0 \quad \text{and}\quad I(0,x)\geq0, \quad x\in \mathcal{O},
		\end{equation}
		where $\mathcal{O} \subset \mathbb{R}^n (n\geq 1)$ is an open bounded set with boundary $\partial \mathcal{O}$ of class $C^2$, $\partial_\nu$ is the outward normal derivative on $\partial \mathcal{O}$ while $W_1(t, x)$ and $W_2(t, x)$ are $L^2(\mathcal{O},\mathbb{R})$-valued Wiener processes. For a detailed biological signification of the remaining positive parameters in the stochastic system \eqref{eq:spdesir}, we refer to \cite{nguyen2019stochastic}.

		We outline that due to Lyapunov method developed by Mao \cite{mao2002environmental}, the approach used to establish well-posedness results for a variety of stochastic epidemic models, governed by coupled stochastic nonlinear differential equations, is now well-understood and the main challenge resides in addressing other problems such as extinction, persistence, stability, stationary distribution and periodic solutions (see e.g. \cite{mehdaoui2022dynamical,jiang2011asymptotic,mehdaoui2023analysisstochastic,khan2022stochastic}). However, in the case of epidemic models governed by coupled stochastic semilinear partial differential equations, the existence of pathwise unique biologically feasible solutions, alone, can be challenging, since the aforementioned approach fails. This is due to the fact that it is based on the existence and uniqueness of local strong solutions, which are not guaranteed. Furthermore, it requires the application of the infinite dimensional Itô rule \cite[Theorem 3.8]{curtain1970ito}, which is not directly applicable to mild solutions.

		In the context of Gaussian noise, the above limitations have been addressed by Nguyen \emph{et al.} \cite{nguyen2019stochastic} in their pioneering work, laying the ground-work for some recently-extended epidemic models developed by Nguyen \emph{et al.} \cite{nguyen2020analysis}, Hu \emph{et al.} \cite{hu2022analysis} and Shao \emph{et al.} \cite{shao2022necessary}. However, in all the aforementioned models, the chosen type of noise is suitable to model only small fluctuations, and it cannot be used to capture sudden environmental changes with discontinuous arrivals. Therefore, an interesting question is to investigate the mathematical well-posedness and biological feasibility of a new class of spatio-temporal stochastic epidemic models, in which Gaussian noise is replaced with a suitable stochastic jump process,  describing such discontinuous random changes. It is worth mentioning that a few researchers have addressed the question of extending deterministic partial differential equations, describing phenomena arising in other fields such as physics and engineering, to the stochastic case, driven by pure-jump Lévy processes. In this regard, we mention the pioneering results of Brze{\'z}niak and Zhu \cite{zhu2016nonlinear}, for the extension of the deterministic beam model, Jiang \emph{et al.}  \cite{jiang2012stochastic} for the extension of the deterministic wave equation, Liang and Gao \cite{liang2014stochastic}, for the extension of the deterministic wave equation with memory, and Bessaih \emph{et al.} \cite{bessaih2015strong} for the extension of equations arising in hydrodynamics.

		Taking the above discussion into account, the focal point of this paper is to address the questions of mathematical well-posedness and biological feasibility for a new class of stochastic epidemic models, governed by coupled stochastic semilinear partial differential equations, which are driven by pure-jump Lévy noise. To the very best of our knowledge, this is the first paper to rigorously address these questions and the only existing results in the literature, when it comes to the stochastic case, are the ones where Gaussian noise has been used \cite{nguyen2019stochastic,nguyen2020analysis,shao2022necessary,hu2022analysis}.

		Let $(\Omega,(\mathcal{F}_t)_{t \geq 0},\mathbb{P})$  be a probability space, with a filtration $\mathbb{F}:=(\mathcal{F}_t)_{t \geq 0}$ satisfying the usual conditions. The model in question is expressed as follows:
		\begin{equation}\label{eq:spdesirlevy}
			\begin{cases}
				\begin{split}
					&\begin{split}
						d S(t, x)&=\left[d_1 \Delta S(t, x)+\Lambda(x)-\mu(x) S(t, x)-F(S(t,x),I(t,x))\right] dt \\
						&+\int_{\mathbb{Z}}  \mathcal{C}_1(z,x) S(t-,x) \;\tilde{N}(dt,dz), \end{split} &\quad  (t,x)\in \mathbb{R}^{+} \times \mathcal{U},\\[1.3ex]
					& \begin{split} 
						dI(t, x)&=\left[d_2 \Delta I(t, x)-\mu(x) I(t, x)-\gamma(x) I(t, x)+{F(S(t,x),I(t,x))}\right] dt\\[1.3ex]
						&+\int_{\mathbb{Z}}  \mathcal{C}_2(z,x) I(t-,x) \;\tilde{N}(dt,dz),\end{split}
					&\quad  (t,x)\in \mathbb{R}^{+} \times \mathcal{U},\\[1.3ex]
					& \begin{split} 
						dR(t,x)&=\left[d_3 \Delta R(t, x)-\mu(x) R(t, x)+{\gamma(x) I(t, x)}\right] dt\\[0.5ex]
						&+\int_{\mathbb{Z}}  \mathcal{C}_3(z,x) R(t-,x)  \;\tilde{N}(dt,dz),\end{split}
					&\quad  (t,x)\in \mathbb{R}^{+} \times \mathcal{U},
				\end{split}
			\end{cases}
		\end{equation}
		equipped with the following homogeneous Neumann boundary conditions:
		\begin{equation}\label{eq:bound}
			\partial_\nu S(t, x)=\partial_\nu I(t, x)=\partial_\nu R(t, x)=0, \quad (t,x)\in \mathbb{R}^{+} \times \partial \mathcal{U},
		\end{equation}
		and the positive initial conditions: 
		\begin{equation}\label{eq:init}
			S(0,x)=S_0(x)\geq 0, \quad I(0,x)=I_0(x)\geq 0, \quad \text{and} \quad R(0,x)=R_0(x)\geq 0,\quad x\in \mathcal{U},
		\end{equation}
		where $\mathcal{U} \subset \mathbb{R}^n (n\geq 1)$ is an open bounded set with smooth boundary $\partial \mathcal{U}$ of class $C^2,$ $\Delta$ denotes the Laplace operator with respect to the spatial variable and $\partial_\nu$ denotes the outward normal derivative on $\partial \mathcal{U}.$

		We mention that the boundary condition \eqref{eq:bound} has been used by several authors for a wide range of  deterministic epidemic models (see e.g. \cite{mehdaoui2022optimal,song2019spatial,zhou2019optimal,mehdaouianalysis} and the references therein). From the biological point of view, this condition is suitable when the population under study is assumed to be at lock-down. That is, individuals are not allowed to enter nor cross the boundary of the spatial domain.  On the other hand, the motivation behind the initial conditions \eqref{eq:init} is purely biological, since we are dealing with populations' densities. For more details on this subject, we refer the reader to the interesting monograph given in \cite{okubo2001diffusion}.

		In this paper, the discontinuous random changes are mathematically modeled by a pure-jump Lévy process. Namely, given a measure space $(\mathbb{Z},\mathcal{Z},\nu),$ we denote by $\tilde{N},$ the compensated Poisson random measure, defined by:
		$$
		\tilde{N}((0, t] \times \mathbb{S}):=N((0, t] \times \mathbb{S})-t \nu(\mathbb{S}), \quad \forall t> 0, \quad \forall \mathbb{S} \in \mathcal{Z},
		$$
		where $N$ is a Poisson random measure, with a corresponding intensity measure $\nu(.),$ which is assumed to be $\sigma$-finite. For a detailed terminology on compensated Poisson random measures, we refer the reader to the monographs given in \cite{ken1999levy,applebaum2009levy}.

		In System \eqref{eq:spdesirlevy}, $d_1, d_2$ and $d_3$ are the positive spatial diffusion rates of the susceptible, infected and recovered individuals, $\Lambda,$ $\mu,$ and $\gamma$ are essentially bounded positive functions, depending on the spatial variable, and their biological signification is the same one considered for Model \eqref{eq:sirsde}. For $i \in \{1,2,3\},$ the functions $
		\mathcal{C}_i\; : \mathbb{Z} \times \mathcal{U} \times \Omega \longrightarrow \mathbb{R},$ model  the intensities of the jump noise and are assumed to be $(\mathcal{Z}\times \mathcal{F}_t)$-measurable and essentially bounded with respect to $(z,x) \in \mathbb{Z}\times \mathcal{U}$.

		The disease incidence rate is mathematically modeled by the function $F,$ whose expression depends on the incorporated biological characteristics. In this paper, we distinguish two types, based on the following mathematical properties:
		\begin{enumerate}[label={(\textbf{P}\arabic*)}]
			\item $F:[0,+\infty)^2\longrightarrow [0,+\infty)$ is a globally Lipschitz function. That is, \label{P1}
			$$
			\exists L_F>0, \;\forall s_1,i_1,s_2,i_2 \in [0,+\infty), \quad \vert F(s_1,i_1)-F(s_2,i_2) \vert \leq L_F{\left(\vert s_1-s_2 \vert^2+\vert i_1-i_2\vert^2\right)}^{\frac{1}{2}}.
			$$
			\item $F: [0,+\infty)^2\longrightarrow [0,+\infty)$ is a locally Lipschitz  function, satisfying the following growth condition:
			$$
			\exists C_F>0, \quad F(s,i) \leq C_F s.  
			$$\label{P3}
		\end{enumerate}
		\begin{rem}
			Below are some examples corresponding to incidence functions satisfying \ref{P1} and \ref{P3}, respectively:
			\begin{itemize}
				\item \textbf{Standard incidence rate \ref{P1}:} let $x\in \mathcal{U}$ and $\mathbf{D}:=\{(s,i)\in [0,+\infty)^2|\; s+i\neq0\}.$  Then, define
				\begin{equation}\label{eq:standard}
					F(s,i):=\begin{cases}
						\begin{split}
							\beta(x){\dfrac{s i}{s+i}}, \quad &\text{if}\; (s,i) \in \mathbf{D},\\[2ex]
							0, \quad &\text{otherwise}.
						\end{split}
					\end{cases}
				\end{equation}
				\item \textbf{Holling-type and Crowley-Martin incidence rates \ref{P3}:} let $x\in \mathcal{U}.$ Then, define
				\begin{equation}\label{eq:saturated}
					F(s,i)=\beta(x)\dfrac{s i}{1+ai+bsi}.
				\end{equation}
			\end{itemize}
			Here, $\beta$ is an essentially bounded positive function, which biologically stands for the disease transmission rate, while $a$ and $b$ are non-negative real constants satisfying $
			a+b\neq 0,$ and
			account for the saturation effect in the disease transmission, which is caused by the protective measures taken by the susceptible population. For more details, we refer for instance to \cite{guan2022bifurcation,holling1959some,crowley1989functional}.
		\end{rem}

		The rest of this paper is arranged as follows. In Section \ref{s2}, we outline some preliminary definitions and provide an abstract formulation of Model \eqref{eq:spdesirlevy}-\eqref{eq:init}. In Section \ref{s3}, we address its mathematical well-posedness, while in Section \ref{section4} we deal with its biological feasibility. On the other hand, in Section \ref{section5}, we provide the outcome of the conducted numerical simulations illustrating the effect of the considered type of stochastic noise, on the spatio-temporal dynamics. At last, in Section \ref{section6}, we state some conclusions,  discuss some possible extensions, and briefly outline some open problems.
		\section{Preliminaries and abstract formulation}\label{s2} 
		Dealing with the mathematical well-posedness and the biological feasibility of Model \eqref{eq:spdesirlevy}-\eqref{eq:init} requires its reformulation in an abstract compact form, which is expressed by a stochastic semilinear evolution equation.

		For $p \in (1,\infty],$ we begin by setting $$\mathbb{L}^p(\mathcal{U}):=(L^p(\mathcal{U}))^3 \quad  \text{and}\quad \mathbb{W}^{2,p}(\mathcal{U}):=(W^{2,p}(\mathcal{U}))^3,$$  
		as the standard product Lebesgue and Sobolev spaces (see e.g. \cite[Chapter 5, p. 245-247]{evans2022partial}).
		
		For the particular case $p=2,$ we set 
		$$
		\mathbb{H}^2(\mathcal{U}):=\mathbb{W}^{2,2}(\mathcal{U}) \quad \text{and} \quad \mathbb{H}:=\mathbb{L}^2(\mathcal{U}).
		$$
		The Hilbert space $\mathbb{H}$ is equipped with its scalar-product-induced norm:
		$$
		\Vert u \Vert_{\mathbb{H}}^2:=\langle u,u \rangle_{\mathbb{H} \times \mathbb{H}}:=\sum_{i=1}^3 \langle u_i,u_i \rangle_{L^2(\mathcal{U}) \times L^2(\mathcal{U})}:=\sum_{i=1}^3 \int_\mathcal{U} u_i^2(x)\;dx.
		$$
		In order to unify the notation, overall throughout the paper, we adopt the following setting: $$u\overset{\Delta}{=}(u_1,u_2,u_3) \overset{\Delta}{=}(S,I,R).$$ 
		Now, consider the following linear operator: 
		\begin{align*}
			\mathcal{A} :\; &\mathcal{D}(\mathcal{A}) \subset \mathbb{L}^p(\mathcal{U}) \longrightarrow \mathbb{L}^p(\mathcal{U}) \\
			&u \mapsto (d_1 \Delta u_1,d_2 \Delta u_2,d_3 \Delta u_3), 
		\end{align*}
		defined on the domain:
		\begin{equation}\label{domneum}
			\mathcal{D}(\mathcal{A}):=\{ u \in \mathbb{W}^{2,p}(\mathcal{U})|\quad \partial_\nu u_i=0,\; \text{on}\; \partial \mathcal{U}, \; \forall i \in \{1,2,3\}\},
		\end{equation}
		equipped with the following graph norm: 
		$$\Vert u \Vert_{\mathcal{D}(\mathcal{A})}:=\Vert \mathcal{A} u \Vert_{\mathbb{L}^p(\mathcal{U})}+\Vert  u \Vert_{\mathbb{L}^p(\mathcal{U})}, \quad \forall u \in \mathcal{D}(\mathcal{A}).$$
		As a direct consequence of Hille-Yosida theorem \cite[Theorem 3.1]{pazy2012semigroups}, it is known that the operator $\mathcal{A}$ generates in $\mathbb{L}^p(\mathcal{U})\;(p \in (1,\infty]),$ an analytic $C_0$-semigroup $(\mathcal{S}_p(t))_{t\geq 0}$. Moreover, it holds that (see e.g. \cite[Chapter 4, p. 109]{cerrai2001second})
		$$
		\mathcal{S}_p(t)u=\mathcal{S}_q(t)u, \quad \forall u \in \mathbb{L}^{p}(\mathcal{U}) \cap  \mathbb{L}^{q}(\mathcal{U}), \quad \forall p,q \in (1,\infty].
		$$
		Additionally, the operator 
		\begin{equation}\label{eq:regula}
			\begin{split}
				(\mathbb{L}^p(\mathcal{U}), &\Vert.\Vert_{\mathbb{L}^p(\mathcal{U})}) \longrightarrow (\mathcal{D}(\mathcal{A}),\Vert. \Vert_{\mathcal{D}(\mathcal{A})})\\
				& u \mapsto \mathcal{S}_p(t) u,
			\end{split}
		\end{equation}
		is bounded for $t>0$ (see e.g. \cite[Chapter 4, p. 109]{cerrai2001second}).
		
		Henceforth, for $p \in (1,\infty],$ the semigroup $(\mathcal{S}_p(t))_{t\geq 0}$ will be simply denoted $(\mathcal{S}(t))_{t \geq 0}$ and we shall omit the explicit notation.

		For $x \in \mathcal{U},$ consider the following functions defined on $\mathbb{R}^3$ by:
		$$
		\begin{cases}
			f_1(s,i,r):=\Lambda(x)-F(s,i)-\mu(x) s,\\
			f_2(s,i,r):=F(s,i)-(\mu(x) +\gamma(x))i,\\
			f_3(s,i,r):=\gamma(x) i - \mu(x) r.
		\end{cases}
		$$
		Additionally, consider the following functions defined on $\mathbb{Z} \times \mathcal{U} \times \mathbb{R}$ by:
		$$
		\begin{cases}
			g_1(z,x,s):=\mathcal{C}_1(z,x) s,\\
			g_2(z,x,i):=\mathcal{C}_2(z,x) i,\\
			g_3(z,x,r):=\mathcal{C}_3(z,x) r.
		\end{cases}
		$$
		Then, define the operator $\mathcal{F}$ by:
		\begin{itemize}
			\item In the case of incidence functions satisfying \ref{P1}:
			\begin{align*}
				\mathcal{F}\;:\; &\mathbb{H}\longrightarrow \mathbb{H}\\
				& u \mapsto \textbf{f}(u).
			\end{align*}
			\item In the case of incidence functions satisfying \ref{P3}:
			\begin{align*}
				\mathcal{F}\;:\; &\mathbb{L}^\infty(\mathcal{U})\longrightarrow \mathbb{L}^\infty(\mathcal{U})\\
				& u \mapsto \textbf{f}(u).
			\end{align*}
		\end{itemize}
		Additionally, we define 
		\begin{itemize}
			\item In the case of incidence functions satisfying \ref{P1}:
			\begin{align*}
				\mathcal{G}\;:&\; \mathbb{Z} \times \mathbb{H}\longrightarrow \mathbb{H}\\
				& (z,u) \mapsto (\textbf{g}_1(z,u_1),\textbf{g}_2(z,u_2),\textbf{g}_3(z,u_3)).
			\end{align*}
			\item In the case of incidence functions satisfying \ref{P3}:
			\begin{align*}
				\mathcal{G}\;:&\; \mathbb{Z} \times \mathbb{L}^\infty(\mathcal{U})\longrightarrow \mathbb{L}^\infty(\mathcal{U})\\
				& (z,u) \mapsto (\textbf{g}_1(z,u_1),\textbf{g}_2(z,u_2),\textbf{g}_3(z,u_3)),
			\end{align*}
		\end{itemize}
		where $\forall i \in \{1,2,3\}:$ 
		\begin{align*}
			\textbf{f}_i(u):\;&\mathcal{U}\; \longrightarrow \mathbb{R}\\
			& x \mapsto f_i(u_1(x)\vee 0,u_2(x)\vee 0,u_3(x)\vee 0), 
		\end{align*} 
		\begin{align*}
			\textbf{g}_i(z,u_i):\;&\mathcal{U} \longrightarrow \mathbb{R}\\
			& x \mapsto g_i(z,x,u_i(x)\vee 0),
		\end{align*}
		and 
		$$a\vee b:=\max\{a,b\},\quad \forall a,b \in \mathbb{R}.$$
		With the above setting, it is clear that the  operator $\mathcal{F}$ satisfies the following assertions: 
		\begin{itemize}
			\item In the case of incidence functions satisfying \ref{P1}: the operator $\mathcal{F}$ is globally Lipschitz.
			\item In the case of incidence functions satisfying \ref{P3}: the operator $\mathcal{F}$ is locally Lipschitz and satisfies the following growth condition:
			\begin{equation}\label{eq:growthoper}
				\exists C_{\mathcal{F}}>0, \quad \Vert \mathcal{F}(u) \Vert_{\mathbb{L}^\infty(\mathcal{U})} \leq C_{\mathcal{F}} \Vert u \Vert_{\mathbb{L}^\infty(\mathcal{U})}, \quad \forall u \in \mathbb{L}^\infty(\mathcal{U}).
			\end{equation}
		\end{itemize}
		
		For simplicity, henceforth, we shall denote by $L_\mathcal{F}$ the local/global Lipschitz constant corresponding to the operator $\mathcal{F}$ in both cases.
		
		Let $$u_0\overset{\Delta}{=}(S_0,I_0,R_0).$$ 
		Based on the preceding definitions, Model \eqref{eq:spdesirlevy}-\eqref{eq:init} can be equivalently expressed in the following abstract compact form: 
		\begin{equation}\label{eq:abs}
			\begin{cases}
				\displaystyle du(t)=\left(\mathcal{A}u(t)+\mathcal{F}(u(t))\right)dt+ \int_{\mathbb{Z}} \mathcal{G}(z,x,u(t-)) \tilde{N}(dt,dz),\\
				u(0)=u_0.
			\end{cases}
		\end{equation}

		Next, we proceed to rigorously precise the types of solutions to Problem \eqref{eq:abs}, that will be investigated in the sequel. To this end, we begin by stating some basic terminology.
		
		First, given $T>0$, we define the $\sigma$-field of progressively measurable subsets by:
		$$
		\mathcal{B}_\mathcal{F}:=\{A \subset [0,T] \times \Omega, \quad A \cup ([0,t] \times \Omega) \in \mathcal{B}([0,t]) \times \mathcal{F}_t,\quad \forall t \in [0,T]\},
		$$
		where $\mathcal{B}([0,t])$ stands for the Borel $\sigma$-algebra generated by 
		intervals of the form $([0,t])_{t\geq0}$.\\\\
		Secondly, we denote by $\mathbb{P}$ the $\sigma$-field generated by real-valued $\mathcal{F}_t$-adapted processes $u:\; \Omega \times \mathbb{R}^+ \longrightarrow \mathbb{R}$, which are  left-continuous, and by $\overline{\mathbb{P}}$ the $\sigma$-field generated by real-valued functions $u:\; \Omega \times \mathbb{R}^+ \times \mathbb{Z} \longrightarrow \mathbb{R}$, satisfying the following assertions:
		\begin{itemize}
			\item For every $t>0$, the function $u_t:\; \Omega \times \mathbb{Z} \ni(\omega, z) \mapsto u_t(\omega,z):=u(t,\omega,z) \in \mathbb{R}$ is $(\mathcal{F}_t \otimes \mathcal{Z}  /\mathcal{B}(\mathbb{R}))$-measurable.
			\item For every $(\omega, z) \in \Omega \times \mathbb{Z}$,  the function $u_{\omega,z}:\; \mathbb{R}_{+} \ni t \mapsto u_{\omega,z}(t):=u(t,\omega,z) \in \mathbb{R}$ is left-continuous.
		\end{itemize}
		Finally, we denote by $\mathbb{B}(\mathcal{X}),$ the $\sigma$-field generated by open sets in a Banach space $\mathcal{X}$.
		\begin{defi}
			A stochastic process $u: \;\mathbb{R}^+ \times \Omega \longrightarrow \mathcal{X}$ is said to be progressively measurable, if it is $(B([0,t])\times \mathcal{F}_t/\mathbb{B(\mathcal{X})})$-measurable for all $t\geq 0$. 
		\end{defi}
		\begin{defi}
			A stochastic process $u : \;\mathbb{R}^+ \times \Omega \longrightarrow \mathcal{X},$ is said to be predictable, if the mapping $$u_t:\; \Omega \times \mathbb{Z} \ni(\omega, z) \mapsto u_t(\omega,z):=u(t,\omega,z) \in \mathcal{X},$$ is $(\mathbb{P}/\mathbb{B}(\mathcal{X}))$-measurable.
		\end{defi}
		\begin{defi}
			A mapping $u:\mathbb{R}^+ \times \Omega \times \mathbb{Z} \longrightarrow \mathcal{X}$ is said to be $\mathbb{F}$-predictable (or simply predictable) if it is $(\overline{\mathbb{P}}/\mathbb{B}(\mathcal{X}))$-measurable.  
		\end{defi}

		For $T>0,$  we define $\mathcal{M}_{\mathbb{F}}^T$ as the space of progressively measurable processes
		$$u : \;\mathbb{R}^+ \times \Omega \longrightarrow \mathcal{D}(\mathcal{A}),$$
		such that 
		$$ \displaystyle
		\mathbb{E} \int_{0}^T \Vert u(t) \Vert_{\mathcal{D}(\mathcal{A})}^2\; dt < + \infty.
		$$ 
		On the other hand, we define $\mathcal{M}_{\mathbb{Z}}^T$ as the space of progressively measurable predictable processes  $$u : \;\mathbb{R}^+ \times \Omega \times \mathbb{Z} \longrightarrow {\mathcal{D}(\mathcal{A})},$$
		such that 
		$$
		\displaystyle
		\mathbb{E} \int_{0}^T \int_{\mathbb{Z}} \Vert u(t,z) \Vert_{\mathcal{D}(\mathcal{A})}^2\;\nu(dz) dt < + \infty.
		$$
		Here and overall throughout the paper, $\mathbb{E}$ will stand for the mathematical expectation.

		\begin{defi}\label{eq:defmild}
			Given a time-horizon $(0,T),$ we define a global mild solution to Problem \eqref{eq:abs} as an $\mathbb{H}$-valued stochastic process $(u(t))_{0\leq t\leq T}$, which is $\mathcal{F}_t$-adapted, satisfies the càd-làg property and for which the following assertions are fulfilled:
			\begin{enumerate}
				\item $\mathcal{S}(t-.)\mathcal{F}(u(.)) \in \mathcal{M}_{\mathbb{F}}^t$\quad and $\quad \mathcal{S}(t-.)  \mathcal{G}(z,x,u(.)) \in \mathcal{M}_{\mathbb{Z}}^t, \quad \forall 0< t\leq T.$
				\item The following equation: $$
				\displaystyle u(t)=\mathcal{S}(t) u_0 + \int_0^t \mathcal{S}(t-s) \mathcal{F}(u(s)) ds+ \int_0^t \int_{\mathbb{Z}} \mathcal{S}(t-s) \mathcal{G}(z,x,u(s-)) \tilde{N}(ds,dz),\\
				$$
				holds $\forall 0\leq t\leq T$, $\mathbb{P}$-almost surely ($\mathbb{P}$-a.s).
			\end{enumerate}

			The global mild solution $u$ is said to be pathwise unique, if for every other mild solution $v$, it holds that 
			\begin{equation}\label{eq:pathuniquenessdef}
				\mathbb{P}(u(t)=v(t),\; \forall 0\leq t\leq T)=1.
			\end{equation} 
		\end{defi}
		\begin{defi}\label{eq:strong}
			Given a time-horizon $(0,T),$ we define a global strong solution to Problem \eqref{eq:abs} as a $\mathcal{D}(\mathcal{A})$-valued stochastic process $(u(t))_{0\leq t\leq T}$, which is $\mathcal{F}_t$-adapted, satisfies the càd-làg property and for which the following assertions are fulfilled:
			\begin{enumerate}
				\item $\mathcal{S}(t-.)\mathcal{F}(u(.)) \in \mathcal{M}_{\mathbb{F}}^t$ \quad and $\quad \mathcal{S}(t-.)  \mathcal{G}(z,x,u(.)) \in \mathcal{M}_{\mathbb{Z}}^t, \quad \forall 0< t\leq T.$
				\item The following equation: 
				$$
				\displaystyle u(t)=u_0 +\int_0^t \mathcal{A} u(s) ds+ \int_0^t  \mathcal{F}(u(s)) ds+ \int_0^t \int_{\mathbb{Z}}  \mathcal{G}(z,x,u(s-)) \tilde{N}(ds,dz),\\
				$$
				holds $\forall 0\leq t\leq T$, $\mathbb{P}$-a.s.\\
			\end{enumerate}
			The global strong solution	is said to be pathwise unique if Property \eqref{eq:pathuniquenessdef} is satisfied, given two strong solutions $u$ and $v$.
		\end{defi}
		\section{Mathematical well-posedness of Model \eqref{eq:spdesirlevy}-\eqref{eq:init}}\label{s3}
		Given a time-horizon $(0,T),$ an initial state and input parameters, the mathematical well-posedness of Model \eqref{eq:spdesirlevy}-\eqref{eq:init} relies on the existence of a pathwise unique mild solution to Problem \eqref{eq:abs}, depending continuously on the initial state.

		\subsection{Well-posedness results in the case of incidence functions satisfying \ref{P1}}
		Throughout this subsection, the incidence function is assumed to satisfy \ref{P1} while the initial condition $u_0$ is assumed to belong to the space $\mathbb{H}$. The main tool that we employ to address the existence and pathwise uniqueness of a global mild solution in this case, is Banach fixed point theorem. The idea is to consider a mapping, defined on a suitable Banach space, and whose fixed point characterizes the mild solution.

		Consider the following Banach space: 
		$$
		\mathbf{E}_{T,\lambda}:=\{u: \mathbb{R}^+ \times \Omega \longrightarrow \mathcal{D}(\mathcal{A}), \; \text{progressively measurable, such that}\; \underset{0\leq t \leq T}{\sup} {\mathbb{E} \Vert u(t) \Vert_{\mathcal{D}(\mathcal{A})}}< +\infty\},
		$$
		equipped with the following norm: 
		\begin{equation}\label{norm}
			\Vert u \Vert_{\mathbf{E}_{T,\lambda}}^2:=\underset{0\leq t \leq T}{\sup} e^{-\lambda t} {\mathbb{E} \Vert u(t) \Vert_{\mathcal{D}(\mathcal{A})}^2}, \quad \forall u \in \mathbf{E}_{T,\lambda},
		\end{equation}
		where $\lambda>0$ will be chosen thereafter accordingly, given that the norms $\left(\Vert . \Vert_{\mathbf{E}_{T,\lambda}}\right)_{\lambda \geq 0}$ are equivalent. 
		
		It is clear that the mild solution to Problem \eqref{eq:abs} is characterized by the fixed point of the following mapping:
		\begin{equation*}
			\begin{split}
			\mathcal{L} :\; &(\mathbf{E}_{T,\lambda},\Vert.\Vert_{\mathbf{E}_{T,\lambda}}) \longrightarrow (\mathbf{E}_{T,\lambda},\Vert.\Vert_{\mathbf{E}_{T,\lambda}})\\
			&u \mapsto \mathcal{L}(u):=\mathcal{S}(.)u_0+\int_0^. \mathcal{S}(.-s) \mathcal{F}(u(s)) ds+\int_0^. \int_{\mathbb{Z}} \mathcal{S}(.-s)\mathcal{G}(z,x,u(s-)) \tilde{N}(ds,dz).
			\end{split}
		\end{equation*}
		Before proceeding to prove that the mapping  $\mathcal{L}$ has a unique fixed point, we first establish that it is well-defined, which is assured by the following lemma:
		\begin{lem}\label{eq:firstlemma}
			Assume that $(u(t))_{t\geq0}$ is progressively measurable. Then, the $\mathcal{D}(\mathcal{A})$-valued stochastic processes:\\
			$$\displaystyle \left(\int_0^t \mathcal{S}(t-s)\mathcal{F}(u(s))ds\right)_{t\geq0}$$ and
			$$\displaystyle\left(\int_0^t \int_{\mathbb{Z}} \mathcal{S}(t-s)\mathcal{G}(z,x,u(s-)) \tilde{N}(ds,dz)\right)_{t\geq0},$$ 
			are well-defined and progressively measurable.
		\end{lem}
		\begin{proof}
			First, note that since $\mathcal{F}$ is continuous, then it is $(\mathbb{B}(\mathcal{D}(\mathcal{A}))/\mathbb{B}(\mathcal{D}(\mathcal{A})))$-measurable.
			On the other hand, since $(\mathcal{S}(t))_{t\geq0}$ is strongly continuous; then $\mathcal{S}(.-s)\mathcal{F}(u(.))$ is progressively measurable, as a composing of  progressively measurable processes. Thus, as a consequence of Fubini's theorem, the stochastic process $$\left(\displaystyle\int_0^t \mathcal{S}(t-s)\mathcal{F}(u(s))ds\right)_{t\geq0},$$ 
			is $\mathcal{F}_t$-adapted. By using the same preceding arguments, it can be proved that the stochastic process $$\displaystyle \left(\int_0^t \int_{\mathbb{Z}} \mathcal{S}(t-s)\mathcal{G}(z,x,u(s-)) \tilde{N}(ds,dz)\right)_{t\geq0},$$
			is also $\mathcal{F}_t$-adapted.
			
			Next, we claim that
			$$
			\displaystyle\int_0^t \mathcal{S}(t-s)\mathcal{F}(u(s))ds, \int_0^t \int_{\mathbb{Z}} \mathcal{S}(t-s)\mathcal{G}(z,x,u(s-)) \tilde{N}(ds,dz)
			\in \mathcal{D}(\mathcal{A})),\quad \forall t>0.$$
			This is established primarily by employing some properties of the resolvent operator, defined as follows:
			\begin{equation}\label{resolvent}
				\forall \theta>0, \quad \mathcal{R}:=\left(-\mathcal{A}+\theta \mathcal{I}\right)^{-1},
			\end{equation}
			where $\mathcal{I}$ denotes the identity operator.\\ 
			Note that $\mathcal{R}$ is compact. Moreover, it holds that (see e.g. \cite[Proof of Theorem 3.1]{pazy2012semigroups})
			\begin{equation}\label{rangereso}
				\mathbf{R}(\mathcal{R})=\mathcal{D}(\mathcal{A}),
			\end{equation}
			where $\mathbf{R}(.)$ stands for the range of a given operator.
			
			With Property \eqref{rangereso} in mind, it suffices to prove that 	$\exists \chi,\overline{\chi} \in \mathcal{D}(\mathcal{R})$
$$
				\displaystyle	
				\int_0^t \mathcal{S}(t-s)\mathcal{F}(u(s))ds=\mathcal{R} \chi,$$
				\text{and}
				$$
				\displaystyle
				\int_0^t \int_{\mathbb{Z}} \mathcal{S}(t-s)\mathcal{G}(z,x,u(s-)) \tilde{N}(ds,dz)=\mathcal{R} \overline{\chi}.
$$
			Now, since the operator $\mathcal{R}$ is closed, by Hille's theorem (see e.g. \cite[Theorem 3.7.12]{hille1948functional}), the operator $\mathcal{R}$ and the integral operator commute. With this in mind, we acquire that   
			\begin{align*}
				\mathcal{R} \int_0^t \mathcal{A} \mathcal{S}(t-s)\mathcal{F}(u(s))ds&=\int_0^t \mathcal{R} \mathcal{A} \mathcal{S}(t-s)\mathcal{F}(u(s))ds\\
				&=\int_0^t \theta \mathcal{R} \mathcal{S}(t-s) \mathcal{F}(u(s))ds-\int_0^t \mathcal{S}(t-s) \mathcal{F}(u(s))ds\\
				&=\mathcal{R}  \int_0^t \theta\mathcal{S}(t-s) \mathcal{F}(u(s))ds-\int_0^t \mathcal{S}(t-s) \mathcal{F}(u(s))ds.
			\end{align*}
			By the same arguments, it can be shown that 
			\begin{equation*}
				\begin{split}
					\mathcal{R} \int_0^t \int_{\mathbb{Z}} \mathcal{A} \mathcal{S}(t-s)\mathcal{G}(z,x,u(s-)) \tilde{N}(ds,dz)&=\mathcal{R}  \int_0^t \theta\mathcal{S}(t-s) \mathcal{G}(z,x,u(s-)) \tilde{N}(ds,dz)\\
					&-\int_0^t \int_{\mathbb{Z}} \mathcal{S}(t-s) \mathcal{G}(z,x,u(s-)) \tilde{N}(ds,dz).
				\end{split}
			\end{equation*}
			The claim follows by setting 
\begin{equation*}
			\displaystyle \chi:=\int_0^t \left(\theta \mathcal{I}-\mathcal{A}\right)\mathcal{S}(t-s) \mathcal{F}(u(s))ds,
\end{equation*}
			and 
\begin{equation*}
			\displaystyle \overline{\chi}:=\int_0^t \int_{\mathbb{Z}} \left(\theta \mathcal{I}-\mathcal{A}\right)\mathcal{S}(t-s) \mathcal{G}(z,x,u(s-)) \tilde{N}(ds,dz).
\end{equation*}
			Now, note that as a direct consequence of \cite[Theorem 4.14]{rudiger2004stochastic}, the following Itô isometry holds:
			\begin{equation}\label{eq:isometry}
				\mathbb{E} \Bigg \Vert \int_0^t \int_{\mathbb{Z}} \mathcal{S}(t-s)\mathcal{G}(z,x,u(s-)) \tilde{N}(ds,dz) \Bigg \Vert_{\mathcal{D}(\mathcal{A})}^2=  \int_0^t \int_{\mathbb{Z}} \mathbb{E} \Vert \mathcal{S}(t-s)\mathcal{G}(z,x,u(s)) \Vert_{\mathcal{D}(\mathcal{A})}^2 \nu(dz)ds.
			\end{equation}
			On the other hand, we infer that
			\begin{equation}\label{eq:iequality}
				\mathbb{E} \Bigg \Vert \int_0^t  \mathcal{S}(t-s)\mathcal{F}(u(s)) ds \Bigg \Vert_{\mathcal{D}(\mathcal{A})}^2\leq \int_0^t  \mathbb{E} \Vert \mathcal{S}(t-s)\mathcal{F}(u(s)) \Vert_{\mathcal{D}(\mathcal{A})}^2 ds.
			\end{equation}
			Finally, we indicate that the right hand side of Equality \eqref{eq:isometry} and that of Inequality \eqref{eq:iequality} are finite due to the boundedness of the mapping defined by \eqref{eq:regula}. 
			
			This concludes the proof.
		\end{proof}
 
		\begin{rem}\label{eq:firstass}
			Note that Lemma \ref{eq:firstlemma} ensures that $$\mathcal{S}(t-.)\mathcal{F}(u(.)) \in \mathcal{M}_{\mathbb{F}}^t,\; \forall 0< t\leq T.$$
			Furthermore, as a consequence of \cite[Theorem 1.1]{zhu2017maximal}, we further acquire that the process $$\displaystyle \left(\int_0^t \int_{\mathbb{Z}} \mathcal{S}(t-s)\mathcal{G}(z,x,u(s-)) \tilde{N}(ds,dz)\right)_{t \geq 0}$$ satisfies the càd-làg property. Hence, $$\mathcal{S}(t-.)  \mathcal{G}(z,x,u(.)) \in \mathcal{M}_{\mathbb{Z}}^t,\;\forall 0< t\leq T.$$
			Consequently, the first assertion of Definition \ref{eq:defmild} is fulfilled.
		\end{rem}
		
		We are now in a position to state the following first well-posedness result:
		\begin{thm}\label{exismildp1}
			For every $u_0 \in \mathbb{H},$ Problem \eqref{eq:abs} admits a pathwise unique global mild solution, depending continuously on the initial condition. 
		\end{thm}
		\begin{proof}
			With Remark \ref{eq:firstass} and Lemma \ref{eq:firstlemma} in mind, it remains to establish that $\mathcal{L}$ is a contraction from $(\mathbf{E}_{T,\lambda},\Vert.\Vert_{\mathbf{E}_{T,\lambda}})$ into itself. To this end, let $u,v \in \mathbf{E}_{T,\lambda},$ we infer that
			\begin{equation}\label{eq:contra_est}
				\begin{split}
					\Vert \mathcal{L}(u)-\mathcal{L}(v)\Vert_{\mathbf{E}_{T,\lambda}}^2&:=\underset{0\leq t\leq T}{\sup} e^{-\lambda t} \mathbb{E} \Bigg \Vert \int_0^t \mathcal{S}(t-s) (\mathcal{F}(u(s))-\mathcal{F}(v(s))) ds \\[-0.5ex]
					&+\int_0^t \int_{\mathbb{Z}} \mathcal{S}(t-s)(\mathcal{G}(z,x,u(s-))-\mathcal{G}(z,x,v(s-))) \tilde{N}(ds,dz) \Bigg \Vert_{\mathcal{D}(\mathcal{A})}^2 \\[-0.5ex]
					&\leq 2 \underset{0\leq t\leq T}{\sup}\left(T e^{-\lambda t} \int_0^t \mathbb{E} \Vert \mathcal{F}(u(s))-\mathcal{F}(v(s)) \Vert_{\mathbb{H}}^2 ds\right.\\[-0.5ex]
					&+\left. e^{-\lambda t}\int_0^t \int_{\mathbb{Z}} \mathbb{E} \Vert (\mathcal{G}(z,x,u(s))-\mathcal{G}(z,x,v(s))) \Vert_{\mathbb{H}}^2 \nu(dz)ds\right)\\[-0.5ex]
					&\leq 2 \underset{0\leq t\leq T}{\sup}\left( T L_\mathcal{F}^2 e^{-\lambda t} \int_0^t \mathbb{E} \Vert u(s)-v(s) \Vert_{\mathcal{D}(\mathcal{A})}^2 ds\right.\\[-0.5ex]
					&+\left. \max_{i \in \{1,2,3\}} \Vert \mathcal{C}_i \Vert_{L^\infty(\mathbb{Z} \times \mathcal{U})}^2 e^{-\lambda t} \int_0^t \int_{\mathbb{Z}} \mathbb{E} \Vert (u(s)-(v(s)) \Vert_{\mathcal{D}(\mathcal{A})}^2 \nu(dz)ds\right)\\[-0.5ex]
					&\leq 2 T L_\mathcal{F}^2 \underset{0\leq t\leq T}{\sup}  \int_0^t e^{-\lambda (t-s)} \mathbb{E} \Vert u(s)-v(s) \Vert_{\mathcal{D}(\mathcal{A})}^2 e^{-\lambda s} ds\\[-0.5ex]
					&+ 2 \max_{i \in \{1,2,3\}} \Vert \mathcal{C}_i \Vert_{L^\infty(\mathbb{Z} \times \mathcal{U})}^2  \nu(\mathbb{Z}) \underset{0\leq t\leq T}{\sup} \int_0^t e^{-\lambda(t-s)} \mathbb{E} \Vert u(s)-v(s) \Vert_{\mathcal{D}(\mathcal{A})}^2 e^{-\lambda s} ds\\[-0.5ex]
					&\leq 2 e^{-\lambda T} \int_0^T e^{\lambda s} ds\left(T L_M^2+\max_{i \in \{1,2,3\}} \Vert \mathcal{C}_i \Vert_{L^\infty(\mathbb{Z} \times \mathcal{U})}^2  \nu(\mathbb{Z})\right) \underset{0\leq s \leq T}{\sup} e^{-\lambda s} \mathbb{E} \Vert u(s)-v(s) \Vert_{\mathcal{D}(\mathcal{A})} \\
					&\leq \dfrac{2}{\lambda} \left(T L_\mathcal{F}^2+\max_{i \in \{1,2,3\}} \Vert \mathcal{C}_i \Vert_{L^\infty(\mathbb{Z} \times \mathcal{U})}^2  \nu(\mathbb{Z})\right) \Vert u-v \Vert_{\mathbf{E}_{T,\lambda}}^2,
				\end{split}
			\end{equation}
			where we have used the definition of $\mathcal{G},$ Property \eqref{eq:regula}, Isometry \eqref{eq:isometry} and Cauchy-Schwartz inequality.
			
			Now, by choosing $\lambda>0$ large enough such that 
			\begin{equation}\label{eq:lambda}
				\lambda > 4 \left(T L_\mathcal{F}^2+\max_{i \in \{1,2,3\}} \Vert \mathcal{C}_i \Vert_{L^\infty(\mathbb{Z} \times \mathcal{U})}^2  \nu(\mathbb{Z})\right),
			\end{equation}
			it follows that
			$$
			\Vert \mathcal{L}(u)-\mathcal{L}(v)\Vert_{\mathbf{E}_{T,\lambda}} \leq \dfrac{1}{\sqrt{2}} \Vert u-v \Vert_{\mathbf{E}_{T,\lambda}}.
			$$
			Thus, by Banach fixed point theorem, it follows that 
			\begin{equation}\label{eq:first}
				\exists ! \textbf{u} \in \mathbf{E}_{T,\lambda}, \quad \text{such that}\quad (\mathcal{L}(\textbf{u}))(t)=\textbf{u}(t),\quad  \mathbb{P}\text{-a.s,} \; \forall t \in (0,T).
			\end{equation}
			Here, we precise that the uniqueness is in the sense of modifications. That is,
			\begin{equation}\label{eq:uniqueasure}
				\left(\exists v \in \mathbf{E}_{T,\lambda}, \quad \mathcal{L}(v)=v\right) \Longrightarrow \textbf{u}(t)=v(t),\quad \mathbb{P}\text{-a.s}, \; \forall t \in (0,T).
			\end{equation}
			However, to achieve that the uniqueness in the sense of modifications is equivalent to pathwise uniqueness, we still have  to prove that the mild solution has a càd-làg modification. Note that thanks to the global Lipschitzity of the operator $\mathcal{F}$, the process
			\begin{align*}
				(u(t))_{t \geq 0}&\overset{\Delta}{=}\left(\mathcal{S}(t) \textbf{u}_0 + \int_0^t \mathcal{S}(t-s) \mathcal{F}(\textbf{u}(s)) ds+ \int_0^t \int_{\mathbb{Z}} \mathcal{S}(t-s) \mathcal{G}(z,x,\textbf{u}(s-)) \tilde{N}(ds,dz)\right)_{t \geq 0}\\[0.5ex]
				&=:(\mathcal{L}(\textbf{u}))_{t\geq0},
			\end{align*}
			satisfies the càd-làg property. Furthermore, it holds that 
			\begin{equation}\label{eq:second}
				\mathbb{E} \Vert \textbf{u}(t)-u(t) \Vert_{\mathcal{D}(\mathcal{A})}^2=0.
			\end{equation}
			Now, it remains to prove that 
			$$
			\mathcal{L}(u)(t)=u(t), \quad \mathbb{P}\text{-a.s},\; \forall t \in(0,T),
			$$
			which immediately follows by combining \eqref{eq:first} and \eqref{eq:second}.
			
			Thus, the pathwise uniqueness of the mild solution is an immediate consequence of \eqref{eq:uniqueasure} together with the càd-làg property (see e.g.  \cite[Chapter 1]{karatzas1991brownian}).
			
			To conclude the proof, we still have to establish the continuity of the pathwise unique mild solution with respect to the initial condition. To this end, let $u$ and $\hat{u}$ be the pathwise unique mild solutions corresponding to the initial conditions  $u_0$ and $\hat{u}_0,$ respectively. Then, it follows that 
			$$
			\mathbb{E} \Vert u(t)-\hat{u}(t) \Vert_{\mathcal{D}(\mathcal{A})}^2 \leq \Vert u_0-\hat{u}_0 \Vert_{\mathbb{H}}^2+\left(L_\mathcal{F}+\underset{i \in \{1,2,3\}}{\max} \Vert \mathcal{C}_i \Vert_{L^\infty(\mathbb{Z} \times \mathcal{U})} \nu(\mathbb{Z})\right) \int_0^t \mathbb{E} \Vert u(s)-\hat{u}(s) \Vert_{\mathcal{D}(\mathcal{A})}^2 ds.
			$$ 
			By virtue of the integral form of Gronwall's inequality (see e.g. \cite[Appendix B]{evans2022partial}) yields
			$$
			\Vert u-\hat{u} \Vert_{\mathbf{E}_{T,0}}^2  \leq\Vert u_0-\hat{u}_0 \Vert_{\mathbb{H}}^2 \left(1+T e^{\left(L_\mathcal{F}+\underset{i \in \{1,2,3\}}{\max} \Vert \mathcal{C}_i \Vert_{L^\infty(\mathbb{Z} \times \mathcal{U})} \nu(\mathbb{Z}) \right) T}\right).
			$$
			The result follows by the equivalence of the norms $\left(\Vert . \Vert_{\mathbf{E}_{T,\lambda}}\right)_{\lambda \geq 0}$.
		\end{proof}
		\begin{rem}
			Observe that the importance of the parameter $\lambda>0$ relies on the fact that for every assigned values to the biological parameters $\Lambda, \mu, \gamma$ and $\mathcal{C}_i\;(i \in \{1,2,3\}),$ one can always find a suitable parameter $\lambda(\Lambda,\mu,\gamma,\mathcal{C}_1,\mathcal{C}_2,\mathcal{C}_3)>0,$ satisfying Inequality \eqref{eq:lambda}, which guarantees that the mapping $\mathcal{L}$ defines a contraction. This induces the incorporation of variate scenarios, without a restriction on the assigned values to the biological parameters, especially when it comes to numerical simulations.
		\end{rem}
		\subsection{Well-posedness results in the case of incidence functions satisfying \ref{P3}}
		In this subsection, the incidence function is assumed to satisfy \ref{P3}. The Lipschitz property of the operator $\mathcal{F}$ in this case is merely local, which prohibits the direct use of Banach fixed point theorem. In light of this difficulty, we proceed by a standard truncation technique (see e.g. \cite[p. 86]{chow2014stochastic}). 
		Given $M>0,$ choose a $\mathcal{C}^\infty$- function $\Pi_M$ such that
		\begin{equation*}
			\quad \Pi_M(r):=\begin{cases}
				\begin{split}
					&1, &\quad \text{if}\quad 0 \leq r \leq M,\\
					&0,  &\quad \text{if}\quad r>2M.
				\end{split}
			\end{cases}
		\end{equation*}
		Additionally, set 
		\begin{equation*}
			\mathcal{F}^M(u):=\mathcal{F}(\Pi_M(\Vert u \Vert_{\mathbb{L}^{\infty}(\mathcal{U})})u), \quad \forall u \in \mathbb{L}^\infty(\mathcal{U}).
		\end{equation*}
		Now, introduce the following intermediate problem:
		\begin{equation}\label{eq:abstrunc}
			\begin{cases}
				\displaystyle du(t)=\left(\mathcal{A}u(t)+\mathcal{F}^M(u(t)\right) dt+ \int_{\mathbb{Z}} \mathcal{G}(z,x,u(t-)) \tilde{N}(dt,dz),\\
				u(0)=u_0.
			\end{cases}
		\end{equation}
		Then, consider the following Banach space: 
		$$
		\hat{\mathbf{E}}_{T,\lambda}:=\{u: \mathbb{R}^+ \times \Omega \longrightarrow \mathbb{L^\infty}(\mathcal{U}), \; \text{progressively measurable, such that}\; \underset{0\leq t \leq T}{\sup} {\mathbb{E} \Vert u(t) \Vert_{\mathbb{L^\infty}(\mathcal{U})}}< +\infty\},
		$$
		equipped with the following norm: 
		\begin{equation*}
			\Vert u \Vert_{\hat{\mathbf{E}}_{T,\lambda}}^2:=\underset{0\leq t \leq T}{\sup} e^{-\lambda t} {\mathbb{E} \Vert u(t) \Vert_{\mathbb{L}^\infty(\mathcal{U})}^2}, \quad \forall u \in \hat{\mathbf{E}}_{T,\lambda}, \quad \forall \lambda>0.
		\end{equation*}
		Since the operator $\mathcal{F}^M$ is globally Lipschitz, by proceeding analogously to the proofs of Lemma \ref{eq:firstlemma} and Theorem \ref{exismildp1}, it can be proved that Problem \eqref{eq:abstrunc} admits a pathwise unique mild solution, which we shall henceforth denote by $u^M$.

		\begin{thm}\label{exismildp3}
			For every $u_0 \in \mathbb{L}^{\infty}(\mathcal{U}),$ Problem \eqref{eq:abs} admits a pathwise unique global mild solution, depending continuously on the initial condition. 
		\end{thm}
		\begin{proof}
			Let $T>0.$ With the càd-làg property taken into account, consider the following well-defined sequence of stopping times: 
			$$
			\tau_M:=\underset{t \in [0,T]}{\inf}\{\Vert u^M(t) \Vert_{\mathbb{L}^\infty(\mathcal{U})} \geq M\}.
			$$
			Obviously, $(\tau_M)_{M>0}$ is an increasing sequence. Hence, there exists $\tau_\infty$ such that 
			$$
			\underset{M \rightarrow +\infty}{\lim} \tau_M(\omega)=\tau_\infty(\omega),\quad \forall \omega \in \Omega.
			$$ 
			Remark that 
			$$
			\mathcal{F}^M(u^M(t))=\mathcal{F}(u^M(t)), \quad \forall t \in (0,\tau_M).
			$$
			Consequently, to conclude the proof it suffices to prove that $\tau_\infty=+\infty.$\\
			To this end, recall that
			$$
			\displaystyle u^M(t)=\mathcal{S}(t) u_0 + \int_0^t \mathcal{S}(t-s) \mathcal{F}(u^M(s)) ds+ \int_0^t \int_{\mathbb{Z}} \mathcal{S}(t-s) \mathcal{G}(z,x,u^M(s-)) \tilde{N}(ds,dz), \quad \forall t \in (0,\tau_M).
			$$
			Let $t \in (0,\tau_M)$, by employing the growth condition \eqref{eq:growthoper}, we infer that
			$$
			\mathbb{E} \Vert u^M(t) \Vert_{\mathbb{L^\infty}(\mathcal{U})} \leq  \Vert u_0 \Vert_{\mathbb{L^\infty}(\mathcal{U})}+\left(C_{\mathcal{F}}+\underset{i \in \{1,2,3\}}{\max} \Vert \mathcal{C}_i \Vert_{L^\infty(\mathbb{Z} \times \mathcal{U})}\nu(\mathbb{Z})\right) \int_0^t \mathbb{E} \Vert u^M(s) \Vert_{\mathbb{L}^{\infty}(\mathcal{U})} ds.
			$$
			Hence, by virtue of the integral form of Gronwall's inequality (see e.g. \cite[Appendix B]{evans2022partial}), we obtain
			$$
			\mathbb{E} \Vert u^M(t) \Vert_{\mathbb{L}^\infty(\mathcal{U})} \leq \Vert u_0 \Vert_{\mathbb{L^\infty}(\mathcal{U})} \left(1+T e^{\left(C_{\mathcal{F}}+\underset{i \in \{1,2,3\}}{\max} \Vert \mathcal{C}_i \Vert_{L^\infty(\mathbb{Z} \times \mathcal{U})}\nu(\mathbb{Z})\right) T}\right).
			$$
			Now, remark that 
			\begin{align*}
				\mathbb{P} \{\tau_M(\omega)<t\}&=\mathbb{E} \mathbbm{1}_{\left\{\tau_M(\omega)<t\right\} }=\int_{\Omega} \dfrac{\Vert u^M(\tau_M) \Vert_{\mathbb{L^\infty}(\mathcal{U})}}{\Vert u^M(\tau_M) \Vert_{\mathbb{L^\infty}(\mathcal{U})}} \mathbf{1}_{\left\{\tau_M(\omega)<t\right\}} d\mathbb{P}\\
				&\leq \dfrac{1}{M} \Vert u_0 \Vert_{\mathbb{L^\infty}(\mathcal{U})} \left(1+T e^{\left(C_{\mathcal{F}}+\underset{i \in \{1,2,3\}}{\max} \Vert \mathcal{C}_i \Vert_{L^\infty(\mathbb{Z} \times \mathcal{U})} \nu(\mathbb{Z})\right) T}\right).
			\end{align*}
			Define the following decreasing sequence of sets: 
			$$
			A_M:=\{\tau_M<t\}, \quad \forall M\in \mathbb{N}^*.
			$$
			By the below-continuity property, we infer that 
			\begin{align*}
				\mathbb{P}\{\tau_\infty<t\}=\mathbb{P}\left\{\underset{M \in \mathbb{N}^*}{\bigcap} A_M\right\}&=\underset{M \rightarrow +\infty}{\lim} \mathbb{P}(A_M)\\
				&\leq \underset{M \rightarrow +\infty}{\lim}\dfrac{1}{M}  \Vert u_0 \Vert_{\mathbb{L^\infty}(\mathcal{U})} \left(1+T e^{\left(C_{\mathcal{F}}+\underset{i \in \{1,2,3\}}{\max} \Vert \mathcal{C}_i \Vert_{L^\infty(\mathbb{Z} \times \mathcal{U})} \nu(\mathbb{Z})\right) T}\right)\\
				&=0.
			\end{align*}
			Thus, $\tau_\infty=+\infty$ $\mathbb{P}$\text{-a.s.}

			The continuity with respect to the initial condition can be obtained analogously to the proof of Theorem \ref{exismildp1}. Indeed, one can obtain that
			$$
			\Vert u-\hat{u} \Vert_{\mathbb{L}^{\infty}(\mathcal{U})}^2  \leq\Vert u_0-\hat{u}_0 \Vert_{\mathbb{L}^{\infty}(\mathcal{U})}^2 \left(1+T e^{\left(L_\mathcal{F}+\underset{i \in \{1,2,3\}}{\max} \Vert \mathcal{C}_i \Vert_{L^\infty(\mathbb{Z} \times \mathcal{U})} \nu(\mathbb{Z})\right) T}\right).$$
			This concludes the proof.
		\end{proof}
		\subsection{Approximation by means of strong solutions}
		The main result of this subsection concerns the approximation of the pathwise unique global mild solution, given by Theorem \ref{exismildp1}, by a converging sequence of pathwise unique global strong solutions. Let  $\mathcal{R}$ be defined as in \eqref{resolvent}. First, recall that 
		\begin{equation}\label{eq:resolcharac}
			\mathcal{R}u=\int_0^{+\infty} e^{-\theta t}\mathcal{S}(t) u dt, \quad \forall u \in \mathbb{L}^{p}(\mathcal{U}).
		\end{equation}
		Furthermore, by Hille-Yosida theorem (see e.g. \cite[Theorem 3.1]{pazy2012semigroups} and \cite[Lemma 3.2]{pazy2012semigroups}), it holds that 
		$$\Vert \mathcal{R}\Vert  \leq \dfrac{1}{\theta} \quad \text{and}\quad \underset{\theta \rightarrow +\infty}{\lim} \theta\mathcal{R}u=u,\quad \forall u \in \mathbb{L}^{p}(\mathcal{U}).$$
		Now, set
		\begin{equation*}\label{eq:fteta}
			\mathcal{F}^\theta(t):=\theta \mathcal{R} \mathcal{F}(u(t)),\quad \forall t \in (0,T),
		\end{equation*}
		\begin{equation*}\label{eq:fbarre}
			\overline{F}(t):=\mathcal{F}(u(t)),\quad \forall t \in (0,T),
		\end{equation*}
		\begin{equation*}\label{eq:gbarre}
			\overline{G}(t):=\mathcal{G}(z,x,u(t)),\quad \forall t \in (0,T),
		\end{equation*} 
		and
		\begin{equation*}
			u^\theta_0:=\theta \mathcal{R} u_0,
		\end{equation*}
		where $u$ denotes the pathwise unique mild solution to Problem \eqref{eq:abs}.\\
		
		Now, consider the following two intermediate problems:
		\begin{equation}\label{reg}
			\begin{cases}
				\displaystyle dv(t)=\left(\mathcal{A} v(t) + \mathcal{F}^\theta(t)\right)dt+\int_{\mathbb{Z}} \overline{G}(z,x,t-) \tilde{N}(dt,dz),\\
				v(0)=u^\theta_0,
			\end{cases}
		\end{equation}
		and
		\begin{equation}\label{reglim}
			\begin{cases}
				\displaystyle dv(t)=\left(\mathcal{A} v(t) + \overline{F}(t)\right)dt+\int_{\mathbb{Z}} \overline{G}(z,x,t-) \tilde{N}(dt,dz),\\
				v(0)=u_0.
			\end{cases}
		\end{equation}

		We now announce a lemma, which will play a major role in what follows.
		\begin{lem}\label{vtettalemma}
			Assume that the incidence function satisfies \ref{P1} or \ref{P3}, then Problem \eqref{reg} admits a pathwise unique strong solution.
		\end{lem}
		\begin{proof}
			From Theorems \ref{exismildp1} and \ref{exismildp3}, it follows that  Problem \eqref{reg} has a pathwise unique global mild solution, which we shall henceforth denote by $v^\theta.$\\
			Consequently, the equation
			\begin{equation}\label{1}
				\displaystyle v^\theta(r)=\mathcal{S}(r)v^\theta(0)+ \int_0^r \mathcal{S}(r-s)\mathcal{F}^\theta(s)ds+\int_0^r \int_{\mathbb{Z}} \mathcal{S}(r-s)\overline{G}(z,x,s-) \tilde{N}(ds,dz), \quad \forall r \in (0,T),
			\end{equation}
			holds $\mathbb{P}$\text{-a.s.}
			
			By applying the operator $\mathcal{A}$ on both sides of \eqref{1}, and integrating from $0$ to $t\;(T>t>r),$ we obtain  
			\begin{equation}\label{22}
				\int_0^t \mathcal{A} \displaystyle v^\theta(r) dr=\int_0^t \mathcal{A} \mathcal{S}(r)v^\theta(0)dr+\int_0^t \mathcal{A} \int_0^r \mathcal{S}(r-s)\mathcal{F}^\theta(s)dsdr+\int_0^t \mathcal{A}\int_0^r \int_{\mathbb{Z}} \mathcal{S}(r-s)\overline{G}(z,x,s-) \tilde{N}(ds,dz)dr.
			\end{equation}
			Proceeding by the same techniques used in the proof of Lemma \ref{eq:firstlemma}, we deduce that Equality \eqref{22} is well-defined. By using Hille's theorem on the operator $\mathcal{A},$ Equality \eqref{22} becomes 
			$$
			\int_0^t \mathcal{A} \displaystyle v^\theta(r) dr=\int_0^t \mathcal{A} \mathcal{S}(r)v^\theta(0)dr+\int_0^t \int_0^r \mathcal{A}  \mathcal{S}(r-s)\mathcal{F}^\theta(s)dsdr+\int_0^t \int_0^r \int_{\mathbb{Z}} \mathcal{A} \mathcal{S}(r-s)\overline{G}(z,x,s-) \tilde{N}(ds,dz)dr.
			$$
			By using the stochastic version of Fubini's theorem \cite[Lemma A.1.1]{mandrekar2015stochastic} and the commutativity of $(\mathcal{S}(t))_{t\geq0}$ with $\mathcal{A},$ we acquire that
			\begin{align}\label{eq}
				\int_0^t \mathcal{A} \displaystyle v^\theta(r) dr&=\int_0^t  \mathcal{S}(r) \mathcal{A} v^\theta(0)dr+\int_0^t \int_s^t  \mathcal{S}(r-s)\mathcal{A} \mathcal{F}^\theta(s)drds+\int_0^t  \int_{\mathbb{Z}} \int_s^t  \mathcal{S}(r-s) \mathcal{A}\overline{G}(z,x,s-) dr \tilde{N}(ds,dz).
			\end{align}
			Keeping in mind that $$v^\theta(0),\mathcal{F}^\theta(.), \overline{G}(z,x,.) \in \mathcal{D}(\mathcal{A}),\quad \forall (z,x) \in \mathbb{Z} \times \mathcal{U},$$ it holds that 
			\begin{equation}\label{tj1}
				\int_0^t  \mathcal{S}(r) \mathcal{A} v^\theta(0)dr= \mathcal{S}(t)  v^\theta(0)-v^\theta(0),
			\end{equation}
			\begin{equation}\label{tj2}
				\int_s^t  \mathcal{S}(r-s)\mathcal{A} \mathcal{F}^\theta(s)dr=\mathcal{S}(t-s)\mathcal{F}^\theta(s)-\mathcal{F}^\theta(s),
			\end{equation}
			\begin{equation}\label{tj3}
				\int_s^t  \mathcal{S}(r-s) \mathcal{A}\overline{G}(z,x,s-) dr=\mathcal{S}(t-s)\overline{G}(z,x,s-)-\overline{G}(z,x,s-).
			\end{equation}
			By injecting \eqref{tj1}-\eqref{tj3} into \eqref{eq}, we obtain
			\begin{align*}
				\int_0^t \mathcal{A} \displaystyle v^\theta(r) dr&=\underbrace{\mathcal{S}(t)v^\theta(0)+\int_0^t \mathcal{S}(t-s) \mathcal{F}^\theta(s)ds+\int_0^t \int_\mathbb{Z} \mathcal{S}(t-s)\overline{G}(z,x,s-) \tilde{N}(ds,dz)}_{=v^\theta(t)}\\
				&-\left(v^\theta(0)+\int_0^t \mathcal{F}^\theta(s)ds+\int_0^t \int_\mathbb{Z}\overline{G}(z,x,s-) \tilde{N}(ds,dz)\right).
			\end{align*} 
			Thus
			$$
			\displaystyle v^\theta(t)=v^\theta(0)+\int_0^t \mathcal{A} v^\theta(s)ds +\int_0^t \mathcal{F}^\theta(s)ds+\int_0^t \int_{\mathbb{Z}} \overline{G}(z,x,s-) \tilde{N}(ds,dz).
			$$
			This concludes the proof.
		\end{proof}
		\begin{thm}\label{exisstrongp2p3}
			Let $u$ be the pathwise unique mild solution to Problem \eqref{eq:abs} given by Theorems \ref{exismildp1} and \ref{exismildp3}. Then, the following convergences hold:
			\begin{itemize}
				\item If the incidence function satisfies \ref{P1}, then
				$
				\underset{\theta \rightarrow +\infty}{\lim}\Vert u^\theta-u \Vert_{\mathbf{E}_{T,0}}^2=0.
				$
				\item If the incidence function satisfies \ref{P3}, then
				$
				\underset{\theta \rightarrow +\infty}{\lim}\Vert u^\theta-u \Vert_{\hat{\mathbf{E}}_{T,0}}^2=0.
				$
			\end{itemize}
			Here, $u^\theta$ is the pathwise unique global strong solution of the following problem: 
			\begin{equation}\label{approximativeprob}
				\begin{cases}
					\displaystyle du(t)=\left(\mathcal{A}u(t)+\theta \mathcal{R}\mathcal{F}(u(t)\right) dt+\int_{\mathbb{Z}} \mathcal{G}(z,x,u(t-)) \tilde{N}(dt,dz),\\
					u(0)=\theta \mathcal{R} u_0.
				\end{cases}
			\end{equation}
		\end{thm}
		\begin{proof}
			Let $T>0.$ We only treat the case of incidence functions satisfying \ref{P1}. The other case can be treated by similar arguments. From Theorem \ref{exismildp1}, there exists a pathwise unique global mild solution $v$ to Problem \eqref{reglim}. Consequently, 
			$$
			v(t)=\mathcal{S}(t)v(0)+\int_0^t \mathcal{S}(t-s)\overline{F}(s)ds+\int_0^t \int_{\mathbb{Z}} \mathcal{S}(t-s) \overline{G}(z,x,s-) \tilde{N}(ds,dz),\quad \mathbb{P}\text{-a.s,}\; \forall t \in(0,T).
			$$
			On the other hand, let $v^\theta$ be the pathwise unique global strong solution to Problem \eqref{reg}. Then, 
			$$
			v^\theta(t)=\mathcal{S}(t)v^\theta(0)+\int_0^t \mathcal{S}(t-s)\mathcal{F}^\theta(s)ds+\int_0^t \int_{\mathbb{Z}} \mathcal{S}(t-s) \overline{G}(z,x,s-) \tilde{N}(ds,dz),\quad \mathbb{P}\text{-a.s,}\; \forall t \in (0,T).
			$$
			Consequently
			\begin{align}\label{estimvvtetta}
				\Vert v-v^\theta \Vert_{\mathbf{E}_{T,0}}^2&=\bigg \Vert \int_0^. \mathcal{S}(.-s)\left(\overline{F}(s)-\mathcal{F}^\theta(s)\right)ds\bigg\Vert_{\mathbf{E}_{T,0}}^2 \leq  \mathbb{E}\int_0^T \Vert \overline{F}(s)-\mathcal{F}^\theta(s)\Vert_{\mathbb{H}}^2.
			\end{align}
			Now, remark that  
			\begin{equation}\label{fthetaest}
				\underset{\theta \rightarrow + \infty}{\lim}\mathcal{F}^\theta(t)= \overline{F}(t), \quad \forall t \in (0,T).
			\end{equation}
			Moreover
			$$
			\Vert \mathcal{F}^\theta(t) \Vert_{\mathbb{H}}\leq \Vert \mathcal{F}(u(t)) \Vert_{\mathbb{H}}.
			$$
			Thus, by employing the dominated convergence theorem, we obtain 
			$$
			\underset{\theta \rightarrow +\infty}{\lim} \mathbb{E} \int_0^T \Vert \overline{F}(s)-\mathcal{F}^\theta(s)\Vert_{\mathbb{H}}^2ds=0.
			$$
			From Inequality \eqref{estimvvtetta}, it follows that
			\begin{equation}\label{eq:convervtetta}
				\underset{\theta \rightarrow +\infty}{\lim}\Vert v^\theta-v \Vert_{\mathbf{E}_{T,0}}^2=0.
			\end{equation}
			Now, observe that by definitions of $\overline{F}$ and $\mathcal{F}^\theta$ and taking into account the pathwise uniqueness of mild and strong solutions, we obtain that
			\begin{equation}\label{eq:uequalsvps}
				u(t)=v(t), \quad \mathbb{P}\text{-a.s,}\; \forall t \in (0,T), 
			\end{equation}
			and
			\begin{equation}\label{eq:uequalsvps2}
				u^\theta(t)=v^\theta(t), \quad \mathbb{P}\text{-a.s,}\; \forall t \in (0,T). 
			\end{equation}
			Thereby 
			\begin{equation*}
				\underset{\theta \rightarrow +\infty}{\lim}\Vert u^\theta-u \Vert_{\mathbf{E}_{T,0}}^2=0.
			\end{equation*}
			This concludes the proof.
		\end{proof}
		
		\section{Biological feasibility of Model \eqref{eq:spdesirlevy}-\eqref{eq:init}}\label{section4}
		Now that the mathematical well-posedness of Model \eqref{eq:spdesirlevy}-\eqref{eq:init} has been addressed, we still have to address its biological feasibility, which consists of proving the positiveness of the pathwise unique mild solution. Such a property can be established in the framework of partial differential equations, by relying on the maximum principle. On the other hand, in the framework of stochastic partial differential equations driven by Gaussian noise, when the obtained solution is strong, one can proceed by applying an infinite dimensional suitable Itô rule \cite[Theorem 3.8]{curtain1970ito}. In this regard, we mention for instance the references \cite{chow2011explosive,lv2015impacts}, where the positiveness of strong solutions was established, in the aim of proving their blow-up at a finite time. The aforementioned approach is based on the following technical lemma:\\

		\begin{lem}(\cite[Lemma 3.1]{lv2015impacts})\label{lemmachn}
			Let $\epsilon>0$ and consider the following function:
			$$
			\eta_{\epsilon}(r):= \begin{cases}r^2-\dfrac{\epsilon^2}{6}, & r<-\epsilon, \\[2ex] -\dfrac{r^3}{\epsilon}\left(\dfrac{r}{2 \epsilon}+\dfrac{4}{3}\right), & -\epsilon \leq r<0, \\[2ex]
				0, & r \geq 0 .\end{cases}
			$$
			Then, $\eta_\epsilon$ enjoys the following properties:
			\begin{enumerate}
				\item $\eta_\epsilon$ is twice continuously differentiable.
				\item \label{prime} $\eta'_\epsilon(r)\leq 0, \; \forall r \in \mathbb{R}.$
				\item \label{secondder} $\eta''_\epsilon(r)\geq 0, \;\forall r \in \mathbb{R}.$
			\end{enumerate}
		\end{lem}
		
		Since the existence of pathwise unique strong solutions to Problem \eqref{eq:abs} is not achieved, we proceed in a different manner. Namely, we construct a sequence of approximate problems for which the existence of strong solutions is assured by Lemma \ref{vtettalemma}. Then, we proceed by using  an alternative Itô formula, which was initially developed by Zhu and Brze{\'z}niak \cite[Theorem 3.5.3]{zhu2010study}, and then used in several other works (see e.g. \cite{zhu2016nonlinear,liang2014stochastic}). The sequence of approximate problems is chosen such that the estimation of the drift term and the use of Gronwall's inequality, as in \cite[Theorem 3.1]{lv2015impacts} and \cite[Theorem 2.1]{chow2009unbounded} is avoided. Then, we retrieve the desired property, in the case of mild solutions by a convergence argument. \\
		
		First, we construct a sequence of functions  
		$(\zeta_{\epsilon})_{\epsilon>0},$ such that the following assertions are satisfied:
		\begin{enumerate}
			\item \label{pp1} $\zeta_\epsilon$ is continuously differentiable for every $\epsilon>0.$
			\item \label{pp2} $\zeta_\epsilon(r)\eta'_{\epsilon}(r)=0,\; \forall \epsilon>0,\; \forall r\in \mathbb{R}.$
			\item \label{pp3} $\underset{\epsilon \rightarrow 0}{\lim}\eta_\epsilon(r)=r^+,\; \forall r \in \mathbb{R},$ where $r^+:=\max\{0,r\}$ denotes the positive part.
		\end{enumerate}
		The existence of such a sequence is assured by the following lemma:
		\begin{lem}\label{mi}
			Given $\epsilon>0,$ consider the following sequence of functions:
			$$
			\zeta_{\epsilon}(r):= \begin{cases}r, & r\geq\epsilon, \\[2ex] \dfrac{1}{\epsilon^3} r^4-\dfrac{3}{\epsilon^2} r^3+\dfrac{3}{\epsilon} r^2, & 0<r< \epsilon, \\[2ex]
				0, & r \leq 0.\end{cases}
			$$
			Then, $\zeta_\epsilon$ satisfies Assertions \ref{pp1}-\ref{pp3}.
		\end{lem}
		\begin{proof}
			The proof is constructive. Denote by $\mathbf{P}_4[\mathbb{R}]$ the vector space of fourth-degree polynomials with real coefficients.
			Then, for $P \in \mathbf{P}_4[\mathbb{R}],$ define 
			$$
			\zeta_{\epsilon}(r):=\begin{cases}r, & r\geq\epsilon, \\[2ex] P(r), & 0<r< \epsilon,  \\[2ex]
				0, & r \leq 0.\end{cases}
			$$
			Then, clearly 
			$$\zeta_\epsilon(r)\eta'_{\epsilon}(r)=0,\; \forall \epsilon>0,\; \forall r\in \mathbb{R}.$$
			Additionally
			$$\underset{\epsilon \rightarrow 0}{\lim}\zeta_\epsilon(r)=r^+,\; \forall r \in \mathbb{R}.$$
			Now, choose $P \in \mathbf{P}_4[\mathbb{R}]$ such that $\zeta_{\epsilon}$ is continuously differentiable. The result follows by straightforward calculations.
		\end{proof}
		Based on the result of Lemma \ref{mi}, the sequence of approximate problems is constructed as follows. First, we consider the sequence of operators $\mathcal{F}_\epsilon$ and $\mathcal{G}_\epsilon$ defined by:
		\begin{itemize}
			\item In the case of incidence functions satisfying \ref{P1}:
			\begin{align*}
				\mathcal{F}_\epsilon\;:\; &\mathbb{H}\longrightarrow \mathbb{H}\\
				& u \mapsto \textbf{f}_\epsilon(u),
			\end{align*}
			and 
			\begin{align*}
				\mathcal{G}_\epsilon\;:&\; \mathbb{Z} \times \mathbb{H}\longrightarrow \mathbb{H}\\
				& (z,u) \mapsto (\textbf{g}_{\epsilon1}(z,u_1),\textbf{g}_{\epsilon2}(z,u_2),\textbf{g}_{\epsilon3}(z,u_3)).
			\end{align*}
			\item In the case of incidence functions satisfying \ref{P3}:
			\begin{align*}
				\mathcal{F}_\epsilon\;:\; &\mathbb{L}^\infty(\mathcal{U})\longrightarrow \mathbb{L}^\infty(\mathcal{U})\\
				& u \mapsto \textbf{f}_\epsilon(u),
			\end{align*}
			and 
			\begin{align*}
				\mathcal{G}_\epsilon\;:&\; \mathbb{Z} \times \mathbb{L}^\infty(\mathcal{U})\longrightarrow \mathbb{L}^\infty(\mathcal{U})\\
				& (z,u) \mapsto (\textbf{g}_{\epsilon1}(z,u_1),\textbf{g}_{\epsilon2}(z,u_2),\textbf{g}_{\epsilon3}(z,u_3)),
			\end{align*}
		\end{itemize}
		where $\forall i \in \{1,2,3\}:$ 
		\begin{align*}
			\left(\textbf{f}_\epsilon(u)\right)_i:\;&\mathcal{U}\; \longrightarrow \mathbb{R}\\
			& x \mapsto f_i(\zeta_\epsilon(u_1(x)),\zeta_\epsilon(u_2(x)),\zeta_\epsilon(u_3(x))), 
		\end{align*}
		and
		\begin{align*}
			\textbf{g}_{\epsilon i}(z,u_i):\;&\mathcal{U} \longrightarrow \mathbb{R}\\
			& x \mapsto g_i(z,x,\zeta_\epsilon(u_i(x))),
		\end{align*}
		where the notation $(v)_i$ stands for the $i$th component of a given vector-valued function $v.$\\
		Now, we set 
		\begin{equation}\label{eq:ftetaeps}
			\mathcal{F}_\epsilon^\theta(u(t)):= \theta \mathcal{R} \mathcal{F}_\epsilon(u(t)),\quad \forall t \in (0,T),
		\end{equation}
		and
		\begin{equation}
			u^\theta_0:=\theta \mathcal{R} u_0.
		\end{equation}
		Then, consider the following intermediate problem:
		\begin{equation}\label{eq:absinterm}
			\begin{cases}
				\displaystyle du(t)=\left(\mathcal{A}u(t)+\mathcal{F}^\theta_\epsilon(u(t))\right)dt+ \int_{\mathbb{Z}} \mathcal{G}_\epsilon(z,x,u(t-)) \tilde{N}(dt,dz),\\
				u(0)=u^\theta_0.
			\end{cases}
		\end{equation}
		\begin{rem}
			Proceeding by the same approach used in the proof of Lemma \ref{vtettalemma}, it can be shown that Problem \eqref{eq:absinterm} admits a pathwise unique global strong solution.
		\end{rem}
		\begin{lem}\label{eq:lemmaeps}
			Let $u$ be the pathwise unique global mild solution, to Problem \eqref{eq:abs}, given by Theorems \ref{exismildp1} and \ref{exismildp3}. Then, the following convergences hold:
			\begin{itemize}
				\item If the incidence function satisfies \ref{P1}, then
				$
				\underset{\epsilon \rightarrow 0}{\lim}\underset{\theta \rightarrow +\infty}{\lim}\Vert u_\epsilon^\theta-u^+ \Vert_{\mathbf{E}_{T,0}}^2=0.
				$
				\item If the incidence function satisfies \ref{P3}, then
				$
				\underset{\epsilon \rightarrow 0}{\lim}\underset{\theta \rightarrow +\infty}{\lim}\Vert u_\epsilon^\theta-u^+ \Vert_{\hat{\mathbf{E}}_{T,0}}^2=0.
				$
			\end{itemize}
			Here, $u^\theta_\epsilon$ is the pathwise unique global strong solution to Problem \eqref{eq:absinterm}. 
			\begin{proof}
				By using Assertion \ref{pp3} in Lemma \ref{mi}, the proof follows the exact approach used in the proof of Theorem \ref{exisstrongp2p3}. Thus, we omit it here.
			\end{proof}
		\end{lem}
		\begin{thm}
			Assume the incidence function satisfies \ref{P1} or \ref{P3}. Then, the corresponding pathwise unique global mild solution remains almost surely positive.
		\end{thm}
		\begin{proof}
			We only treat the case of incidence functions satisfying \ref{P1}. The other case can be treated based on the same arguments. Consider the following functional
			$$
			\mathcal{J}(u(t)):=\sum_{i=1}^3 \int_\mathcal{U} \eta_\epsilon(u_i(t,x)) dx.
			$$
			Denote $\eta_\epsilon(u(t,x))\overset{\Delta}{=}\left(\eta_\epsilon(u_1(t,x)),\eta_\epsilon(u_2(t,x)),\eta_\epsilon(u_3(t,x))\right).$ By applying Itô formula \cite[Theorem 3.5.3]{zhu2010study}, it follows that 
			\begin{align}\label{itoineq}
				\mathcal{J}(u_\epsilon^\theta(t))&=\mathcal{J}(u_\epsilon^\theta(0))+\int_0^t \langle \mathcal{A}u_\epsilon^\theta+\mathcal{F}^\theta_\epsilon(u_{\epsilon}^\theta(s)),\eta_\epsilon'(u_\epsilon^\theta(s)) \rangle_{\mathbb{H}\times \mathbb{H}}ds \nonumber\\
				&+\int_0^t \int_{\mathbb{Z}} \left( \eta_\epsilon(u_\epsilon^\theta(s)+\mathcal{G}_\epsilon(z,x,u_\epsilon^\theta(s)))-\eta_\epsilon(u_\epsilon^\theta(s))-\langle \mathcal{G}_\epsilon(z,x,u_\epsilon^\theta(s)),\eta_\epsilon'(u_\epsilon^\theta(s))\rangle_{\mathbb{H} \times \mathbb{H}}\right) \nu(dz)ds\nonumber\\
				&+\int_0^t \int_{\mathbb{Z}} \left(\langle \eta_\epsilon(u_\epsilon^\theta(s)+\mathcal{G}_\epsilon(z,x,u_\epsilon^\theta(s)))-\eta_\epsilon(u_\epsilon^\theta(s)),\eta_\epsilon'(u_\epsilon^\theta(s)) \rangle_{\mathbb{H} \times \mathbb{H}}\right) \tilde{N}(dt,dz).
			\end{align}
			We now proceed to estimate each term independently.
			
			On the one hand
			\begin{align}
				\int_0^t \langle \mathcal{A}u_\epsilon^\theta(s),\eta_\epsilon'(u_\epsilon^\theta(s)) \rangle_{\mathbb{H}\times \mathbb{H}}ds&=\sum_{i=1}^3 \int_0^t \int_{\mathcal{U}} d_i \Delta (u_{\epsilon}^\theta(s,x))_i \eta'_\epsilon((u_{\epsilon}^\theta(s,x))_i)dxds \nonumber \\
				&=\sum_{i=1}^3 \int_0^t \int_{\mathcal{U}} -d_i \vert \nabla (u_{\epsilon}^\theta(s,x))_i \vert^2 \eta''_\epsilon((u_{\epsilon}^\theta(s,x))_i)dxds \nonumber\\
				&\leq 0,
			\end{align}
			where we have used Green's formula together with Assertion \ref{secondder}. 
			
			On the other hand
			\begin{align*}
				\int_0^t \langle \mathcal{F}^\theta_\epsilon(u_{\epsilon}^\theta(s)),\eta_\epsilon'(u_\epsilon^\theta(s)) \rangle_{\mathbb{H}\times \mathbb{H}}ds&=\sum_{i=1}^3 \int_0^t \langle (\mathcal{F}^\theta_\epsilon(u_{\epsilon}^\theta(s)))_i,(\eta_\epsilon'(u_\epsilon^\theta(s)))_i \rangle_{L^2(\mathcal{U})\times L^2(\mathcal{U})}ds \\
				&=:I_1+I_2+I_3.
			\end{align*}
			By definition, it holds that 
			\begin{align*}
				I_1&:=\int_0^t \int_\mathcal{U} \big(\eta_\epsilon'((u_\epsilon^\theta(s,x))_1) \theta \left(\mathcal{R} \Lambda\right)(x)\\
				&+\underbrace{\eta_\epsilon'(u_\epsilon^\theta(s,x)) \theta \left(\mathcal{R}  \left(-F(\zeta_{\epsilon}((u_\epsilon^\theta(s))_1),\zeta_{\epsilon}((u_\epsilon^\theta(s))_2))-\mu \zeta_{\epsilon}((u_\epsilon^\theta(s))_1)\right)\right)(x)}_{=0}\big)dx ds \\
				&=\int_0^t \int_\mathcal{U}\eta_\epsilon'(u_\epsilon^\theta(s,x)) \theta\left(\int_0^{+\infty} \mathcal{S}(t) e^{-\theta t}  \Lambda(x) dt\right) dx ds\leq 0,
			\end{align*}
			where we have used Property \eqref{eq:resolcharac}, Assertion \ref{pp2} in Lemma \ref{mi}, Assertion \ref{prime} in Lemma \ref{lemmachn} and the positivity preserving property of the semigroup  $(\mathcal{S}(t))_{t\geq 0}$ (see e.g. \cite[Lemma 3.5.9]{cazenave1998introduction}). 
			
			By the same arguments, we obtain that
			$$
			I_2:=\int_0^t \int_\mathcal{U}\eta_\epsilon'((u_\epsilon^\theta(s,x))_2) \theta \left(\mathcal{R}  \left(F(\zeta_{\epsilon}((u_\epsilon^\theta(s))_1),\zeta_{\epsilon}((u_\epsilon^\theta(s))_2))-(\mu+\gamma) \zeta_{\epsilon}((u_\epsilon^\theta(s))_2)\right)\right)(x)dx ds=0,
			$$
			and 
			$$
			I_3:=\int_0^t \int_\mathcal{U}\eta_\epsilon'((u_\epsilon^\theta(s,x))_3) \left(\theta \mathcal{R} \left(\left(\gamma \zeta_{\epsilon}((u_\epsilon^\theta(s))_2)-\mu\zeta_{\epsilon}(u_\epsilon^\theta(s))_3\right)\right)\right)(x)dx ds=0.
			$$
			Consequently 
			\begin{align}
				\int_0^t \langle \mathcal{F}^\theta_\epsilon(u_{\epsilon}^\theta(s)),\eta'(u_\epsilon^\theta(s)) \rangle_{\mathbb{H}\times \mathbb{H}}ds\leq0.
			\end{align}
			Additionally, we infer that 
			\begin{align*}
				&\int_0^t \int_{\mathbb{Z}} \left( \eta_\epsilon(u_\epsilon^\theta(s,x)+\mathcal{G}_\epsilon(z,x,u_\epsilon^\theta(s,x)))-\eta_\epsilon(u_\epsilon^\theta(s,x))-\langle \mathcal{G}_\epsilon(z,x,u_\epsilon^\theta(s),\eta_\epsilon'(u_\epsilon^\theta(s)))\rangle_{\mathbb{H} \times \mathbb{H}}\right) \nu(dz)ds\\
				&=\int_0^t \int_{\mathbb{Z}} \left( \eta_\epsilon(u_\epsilon^\theta(s,x)+\mathcal{G}_\epsilon(z,x,u_\epsilon^\theta(s,x)))-\eta_\epsilon(u_\epsilon^\theta(s,x))\right)\nu(dz)ds\\
				&-\int_0^t \int_{\mathbb{Z}} \langle \mathcal{G}_\epsilon(z,x,u_\epsilon^\theta(s),\eta_\epsilon'(u_\epsilon^\theta(s)))\rangle_{\mathbb{H} \times \mathbb{H}}\nu(dz)ds\\
				&=:I_4-I_5.
			\end{align*}
			By employing Taylor's expansion, there exists $\alpha \in (0,1)$ such that 
			\begin{align*}
				I_4&=\int_0^t \int_{\mathbb{Z}} \left(\eta_\epsilon'(\alpha (u_\epsilon^\theta(s,x)+\mathcal{G}_\epsilon(z,x,u_\epsilon^\theta(s,x)))+(1-\alpha)(u_\epsilon^\theta(s,x)))\mathcal{G}_\epsilon(z,x,u_\epsilon^\theta(s,x))\right) \nu(dz)ds\\
				&=\sum_{i=1}^3 \int_0^t \int_{\mathbb{Z}} \eta_{\epsilon}'(\alpha ((u_\epsilon^\theta(s,x))_i+\mathcal{C}_i(z,x)\zeta_\epsilon ((u_\epsilon^\theta(s,x))_i))+(1-\alpha)(u_\epsilon^\theta(s,x))_i  \mathcal{C}_i(z,x)  \zeta_\epsilon((u_\epsilon^\theta(s,x))_i)\nu(dz)ds\\
				&\leq 0,
			\end{align*}
			where we have used Assertion \ref{prime}, stated in Lemma \ref{lemmachn}. \\
			By using Assertion \ref{pp2}, stated in Lemma \ref{mi}, we obtain that
			\begin{align*}
				I_5= \sum_{i=1}^3 \int_0^t \int_{\mathbb{Z}}  \mathcal{C}(z,x) \zeta_\epsilon((u_\epsilon^\theta(s,x))_i) \eta_\epsilon'((u_\epsilon^\theta(s,x))_i) \nu(dz)ds=0.
			\end{align*}
			Since by assumption $u_0 \geq 0.$ Then, by the definitions of $\mathcal{J}$ and $\eta_\epsilon,$ we acquire that 
			$$
			\mathcal{J}(u_\epsilon^\theta(0))=0.
			$$
			Thus, from Equality \eqref{itoineq}, we obtain 
			\begin{align}\label{itoineq2}
				\mathcal{J}(u_\epsilon^\theta(t))&\leq \int_0^t \int_{\mathbb{Z}} \left(\langle \eta_\epsilon(u_\epsilon^\theta(s)+\mathcal{G}_\epsilon(z,x,u_\epsilon^\theta(s)))-\eta_\epsilon(u_\epsilon^\theta(s)),\eta_\epsilon'(u_\epsilon^\theta(s)) \rangle_{\mathbb{H} \times \mathbb{H}}\right) \tilde{N}(dt,dz).
			\end{align}
			Evaluating the mathematical expectation on both sides of Inequality \eqref{itoineq2} yields 
			\begin{align*}
				\mathbb{E} \mathcal{J}(u_\epsilon^\theta(t))&\leq 0.
			\end{align*}
			Hence
			$$
			\mathbb{E} \int_\mathcal{U} \eta_{\epsilon}((u_\epsilon^\theta(t,x))_i) dx\leq 0, \quad \forall i \in \{1,2,3\}.
			$$
			By definition of $\eta_\epsilon,$ we acquire that 
			$$
			(u_\epsilon^\theta(t,x))_i \geq 0, \quad \forall i \in \{1,2,3\},\; \mathbb{P}\text{-a.s.}
			$$
			The proof is concluded by using Lemma \ref{eq:lemmaeps}.
		\end{proof}
		\section{Numerical simulations}\label{section5}
		The aim of this section is to exhibit the effect of pure-jump Lévy noise on the spatio-temporal dynamics of SIR-type epidemic models, and provide a comparison to Gaussian noise, which has been used in the previous literature. The spatial domain $\mathcal{U}$ is set to be $(0,6),$ while the time-horizon $(0,80)$ is considered. Additionally, we choose $\mathbb{Z} \subset (0,+\infty)$ such that $\nu(\mathbb{Z})=1$. Moreover, to simulate the start of an epidemic, the following initial conditions are considered:
		$$
		S(0,x)=0.9, \quad I(0,x)=0.1 \quad \text{and}\quad R(0,x)=0, \quad \forall x \in (0,6).
		$$ 
		The performed numerical simulations are based on the operator splitting method for the deterministic case, the Milstein method for the case of Gaussian noise, and Euler's method for the case of pure-jump Lévy noise. We refer to the references \cite{ahmed2020numerical,higham2001algorithmic,protter1997euler} for a detailed description of the aforementioned methods. \\
		
		The deterministic counterpart of Model \eqref{eq:spdesirlevy}-\eqref{eq:init} can be obtained by setting $\mathcal{C}_1=\mathcal{C}_2=\mathcal{C}_3=0.$ On the other hand, for the sake of convenience, we recall below the stochastic counterpart of Model \eqref{eq:spdesirlevy}-\eqref{eq:init} in the case of Gaussian noise.
		
		\begin{equation}\label{eq:spdesirgauss}
			\begin{cases}
				\begin{split}
					&\begin{split}
						d S(t, x)&=\left[d_1 \Delta S(t, x)+\Lambda(x)-\mu(x) S(t, x)-F(S(t,x),I(t,x))\right] dt \\
						&+\sigma_1 S(t,x) \;dW_1(t,x), \end{split} &\quad  (t,x)\in \mathbb{R}^{+} \times \mathcal{U},\\[1.3ex]
					& \begin{split} 
						dI(t, x)&=\left[d_2 \Delta I(t, x)-\mu(x) I(t, x)-\gamma(x) I(t, x)+{F(S(t,x),I(t,x))}\right] dt\\[1.3ex]
						&+\sigma_2 I(t,x) \;dW_2(t,x),\end{split}
					&\quad  (t,x)\in \mathbb{R}^{+} \times \mathcal{U},\\[1.3ex]
					& \begin{split} 
						dR(t,x)&=\left[d_3 \Delta R(t, x)-\mu(x) R(t, x)+{\gamma(x) I(t, x)}\right] dt\\[0.5ex]
						&+\sigma_3 R(t,x) \;dW_3(t,x),\end{split}
					&\quad  (t,x)\in \mathbb{R}^{+} \times \mathcal{U},
				\end{split}
			\end{cases}
		\end{equation}
		equipped with the usual homogeneous Neumann boundary conditions:
		\begin{equation*}
			\partial_\nu S(t, x)=\partial_\nu I(t, x)=\partial_\nu R(t, x)=0, \quad (t,x)\in \mathbb{R}^{+} \times \partial \mathcal{U},
		\end{equation*}
		and the positive initial conditions: 
		\begin{equation*}
			S(0,x)=S_0(x)\geq 0, \quad I(0,x)=I_0(x)\geq 0, \quad \text{and} \quad R(0,x)=R_0(x)\geq 0,\quad x\in \mathcal{U}.
		\end{equation*}
		Here, $W_1(t,x)$ and $W_2(t,x)$ are as considered in Section \ref{intro}, while $\sigma_1,\sigma_2,\sigma_3>0$ denote Gaussian noise intensities.

		In order to incorporate the two types of incidence rates distinguished in this paper, the following two examples are numerically simulated:
		\subsection{First example: standard incidence rate}
		In this example, the chosen incidence function is given by \eqref{eq:standard}, so that \ref{P1} is satisfied. The appointed values to the input parameters in this case are given in Table \ref{tab:assigned1}. 
		\begin{table}[H]
			\centering
			\setlength\extrarowheight{2.5pt}
			\caption{Assigned values to the remaining parameters}
			\label{tab:assigned1}
			\begin{tabular}{|l|l|lll}
				\cline{1-2}
				Parameter&Assigned value  &  &  &  \\ \cline{1-2}
				$\Lambda$& $0.5$  &  &  &  \\ \cline{1-2}
				$\beta$& $0.2$ &  &  &  \\ \cline{1-2}
				$\mu $& $0.3$ & &  &  \\ \cline{1-2}
				$\gamma$& $0.2$ &  &  &  \\ \cline{1-2}
				$\mathcal{C}_1, \mathcal{C}_2, \mathcal{C}_3$& $0.2$ &  &  &  \\ \cline{1-2}
				$\sigma_1, \sigma_2, \sigma_3$& $1$ &  &  &  \\ \cline{1-2}
			\end{tabular}
		\end{table}
		We indicate that the assigned values given in Table \ref{tab:assigned1}, are in the aim of illustrating the effect of pure-jump Lévy noise in the case: 
		$
		\mathcal{R}_0:=\dfrac{\Lambda \beta}{\mu \left(\mu+\gamma\right)}<1,
		$
		where $\mathcal{R}_0$ is the basic reproduction number of Model \eqref{eq:spdesirlevy}-\eqref{eq:init} in the absence of stochastic noise (see e.g. \cite[Section 5]{yang2020basic}). In this case, it is known that the free disease equilibrium point is globally asymptotically stable, meaning that the infection will be eradicated within the population \cite{luo2019global}. Figures \ref{figstandards}-\ref{figstandardr} show the obtained solution in the deterministic case, stochastic case with Gaussian noise and stochastic case with pure-jump Lévy noise. One can observe that pure-jump Lévy noise causes a delay in the eradication of the disease in comparison to the deterministic case and the stochastic case with Gaussian noise. Furthermore, from Fig. \ref{figstandardpaths}, it can be seen that pure-jump Lévy noise results in a discontinuity in the paths of the obtained solution, which captures the massive discontinuous changes in the spatio-temporal dynamics.
		\begin{figure}[H]
			\begin{minipage}{5cm}
				\includegraphics[height=4cm]{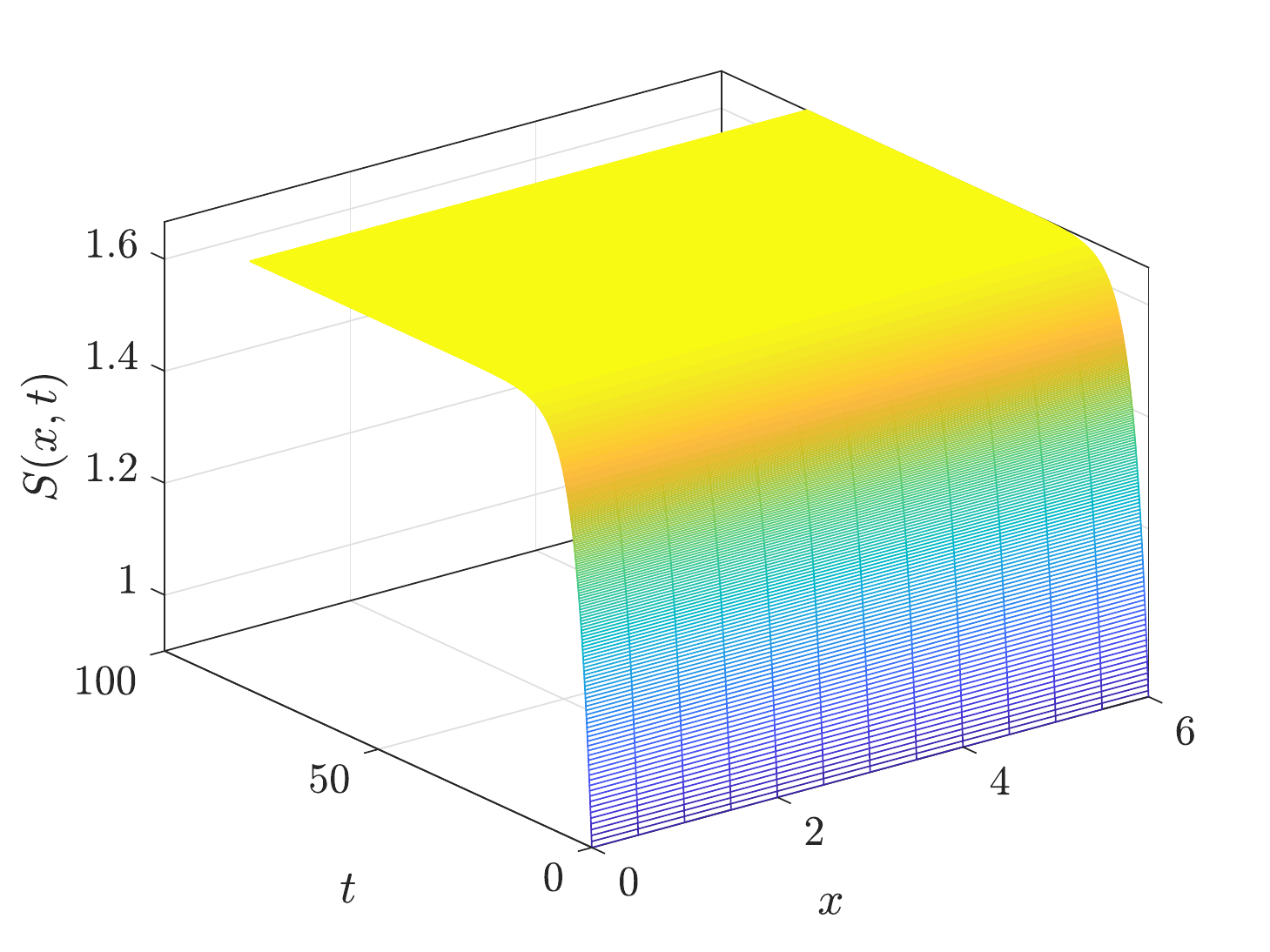}
			\end{minipage}
			\begin{minipage}{5cm}
				\includegraphics[height=4cm]{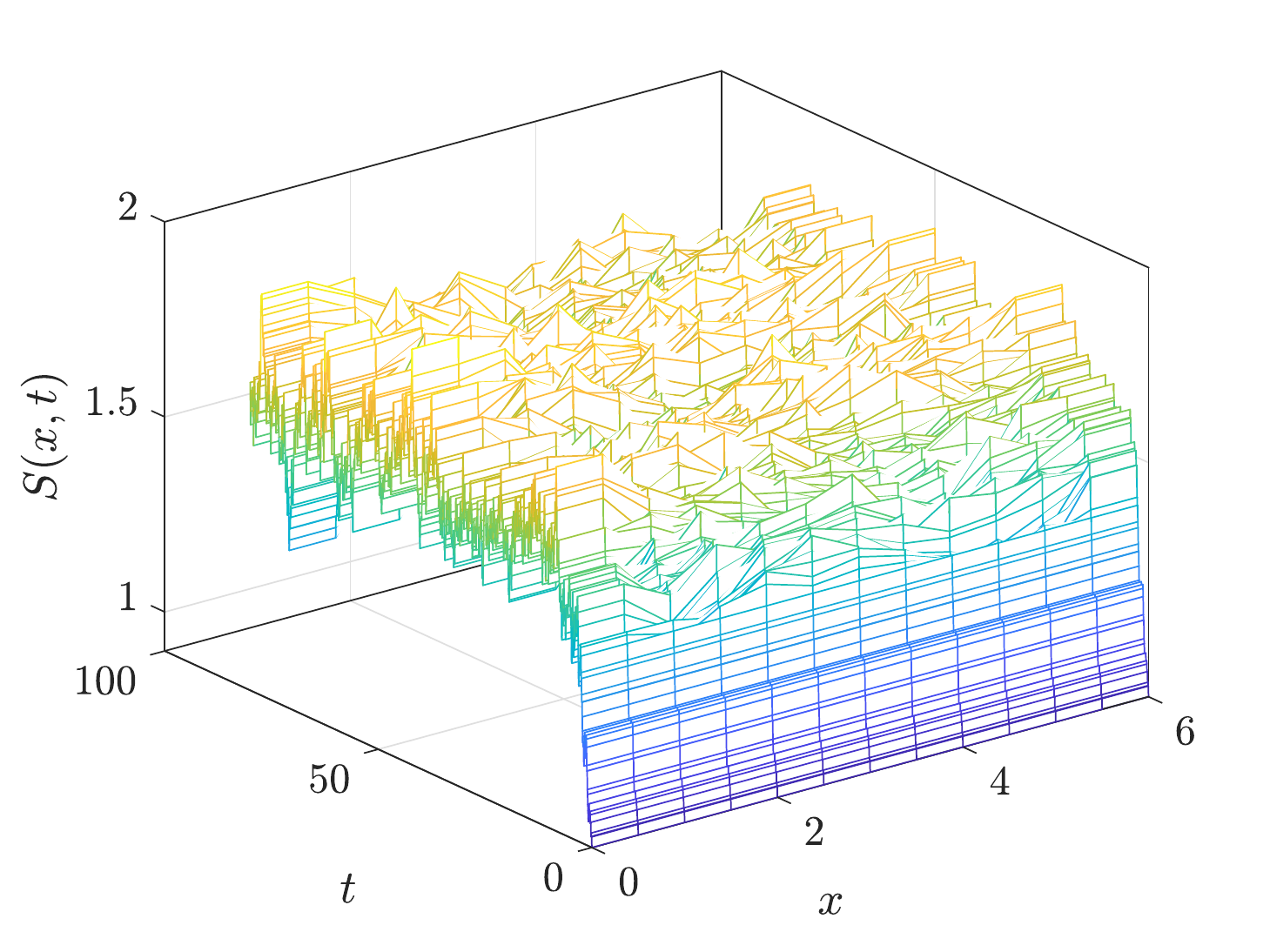}
			\end{minipage}
			\begin{minipage}{5cm}
				\includegraphics[height=4cm]{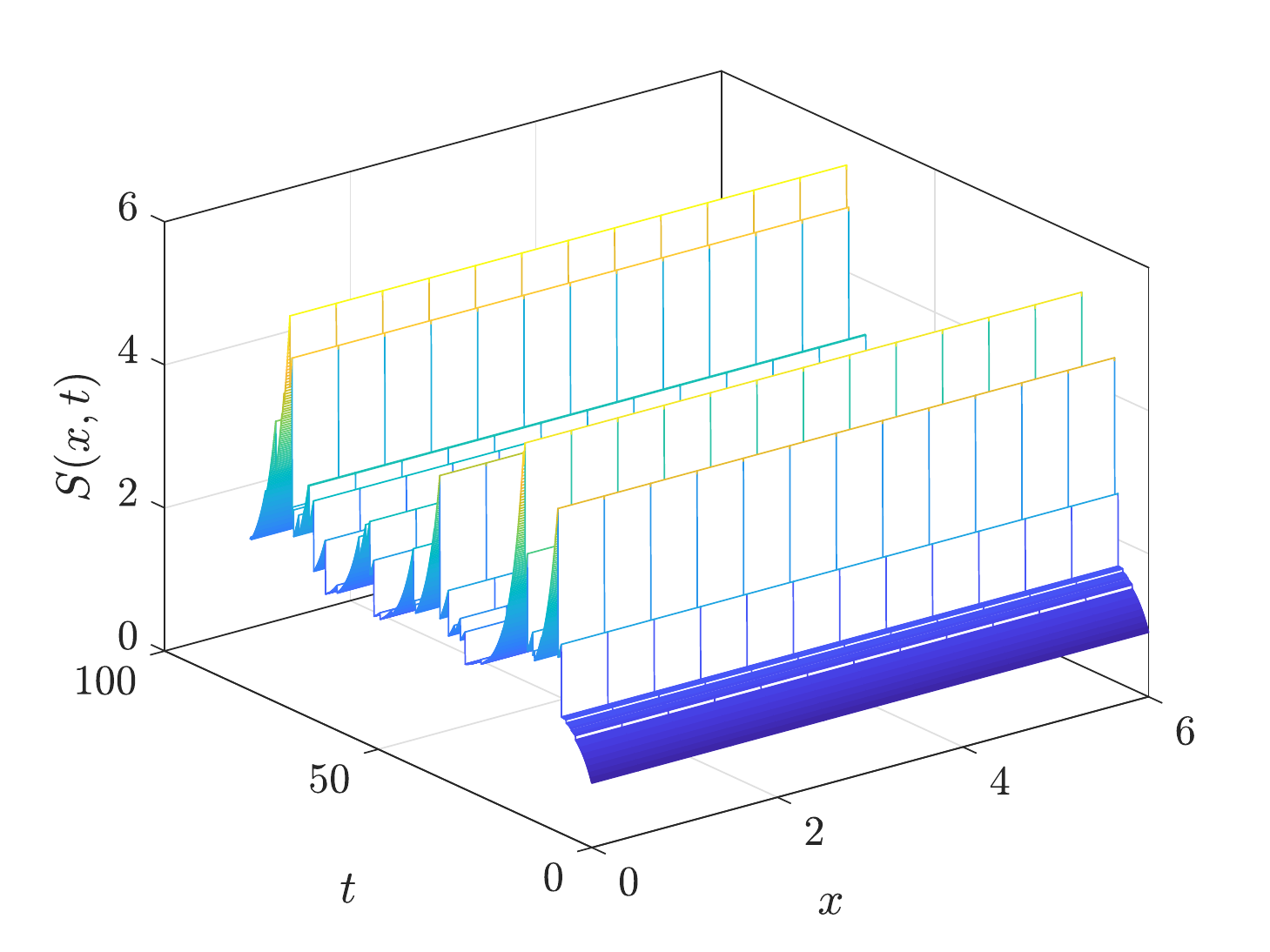}
			\end{minipage}
			\caption{Deterministic susceptible population (left), stochastic susceptible population with Gaussian noise (middle), stochastic susceptible population with pure-jump Lévy noise (right)}
			\label{figstandards}
		\end{figure}
		\begin{figure}[H]
			\begin{minipage}{5cm}
				\includegraphics[height=4cm]{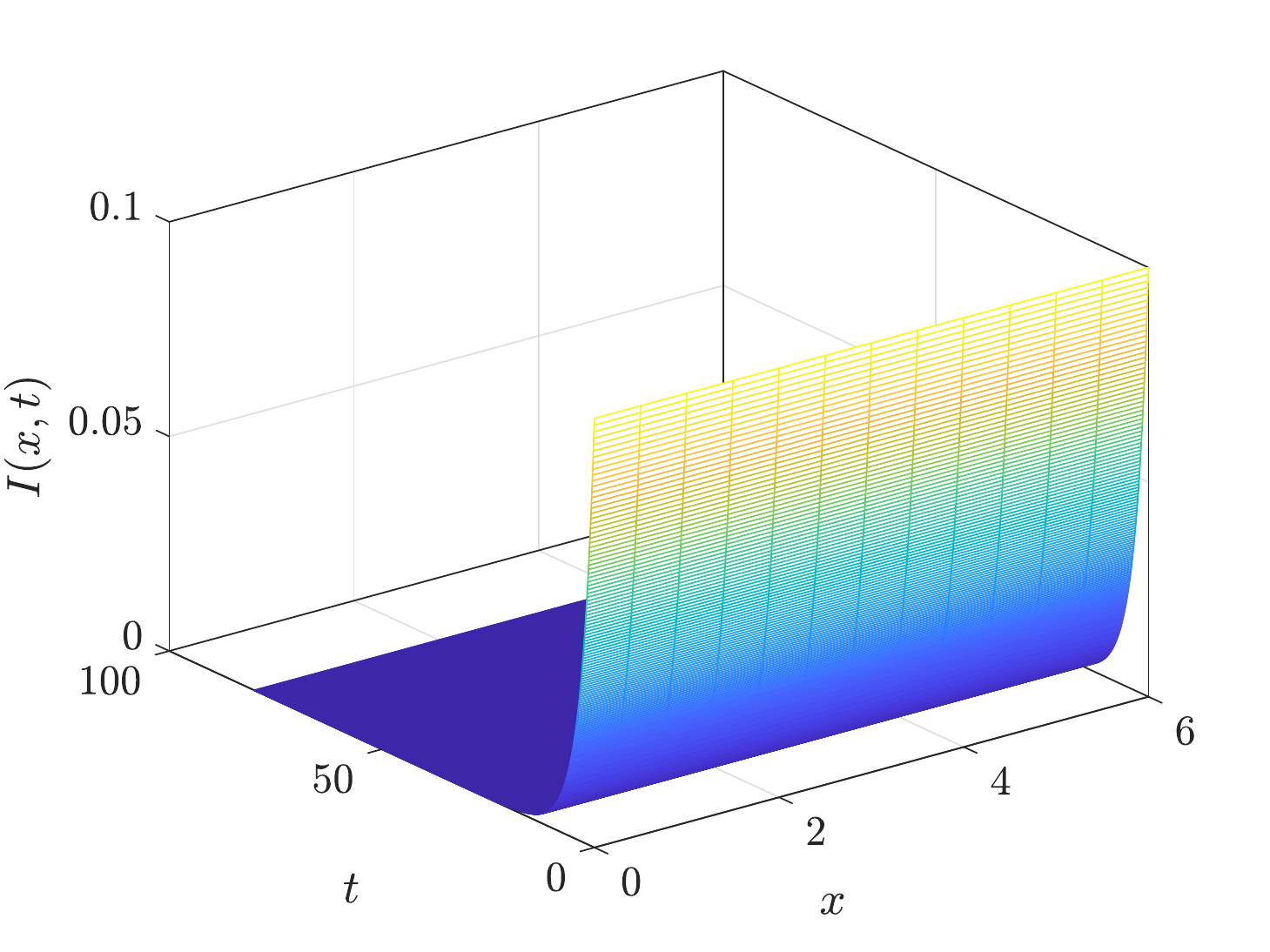}
			\end{minipage}
			\begin{minipage}{5cm}
				\includegraphics[height=4cm]{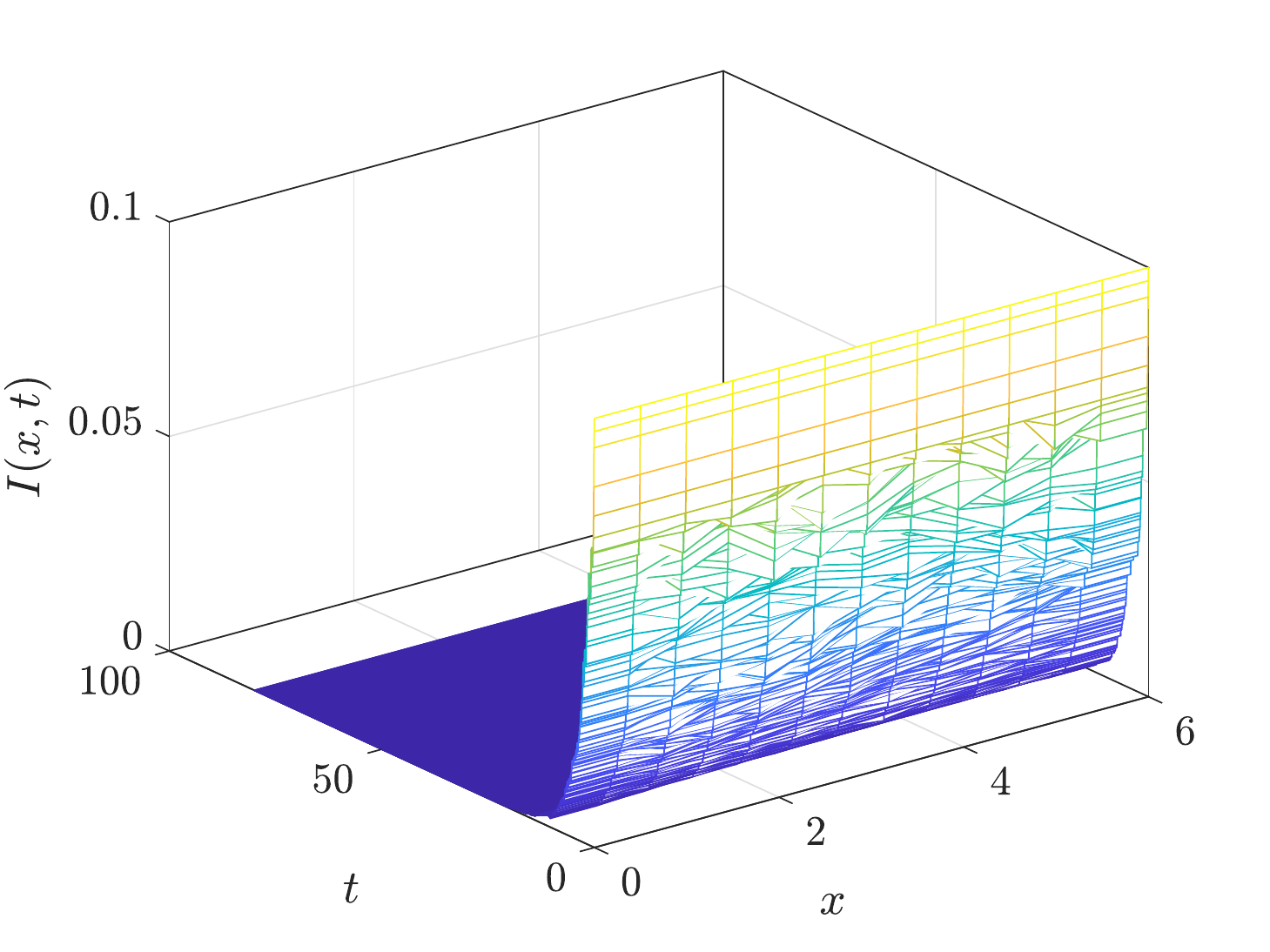}
			\end{minipage}
			\begin{minipage}{5cm}
				\includegraphics[height=4cm]{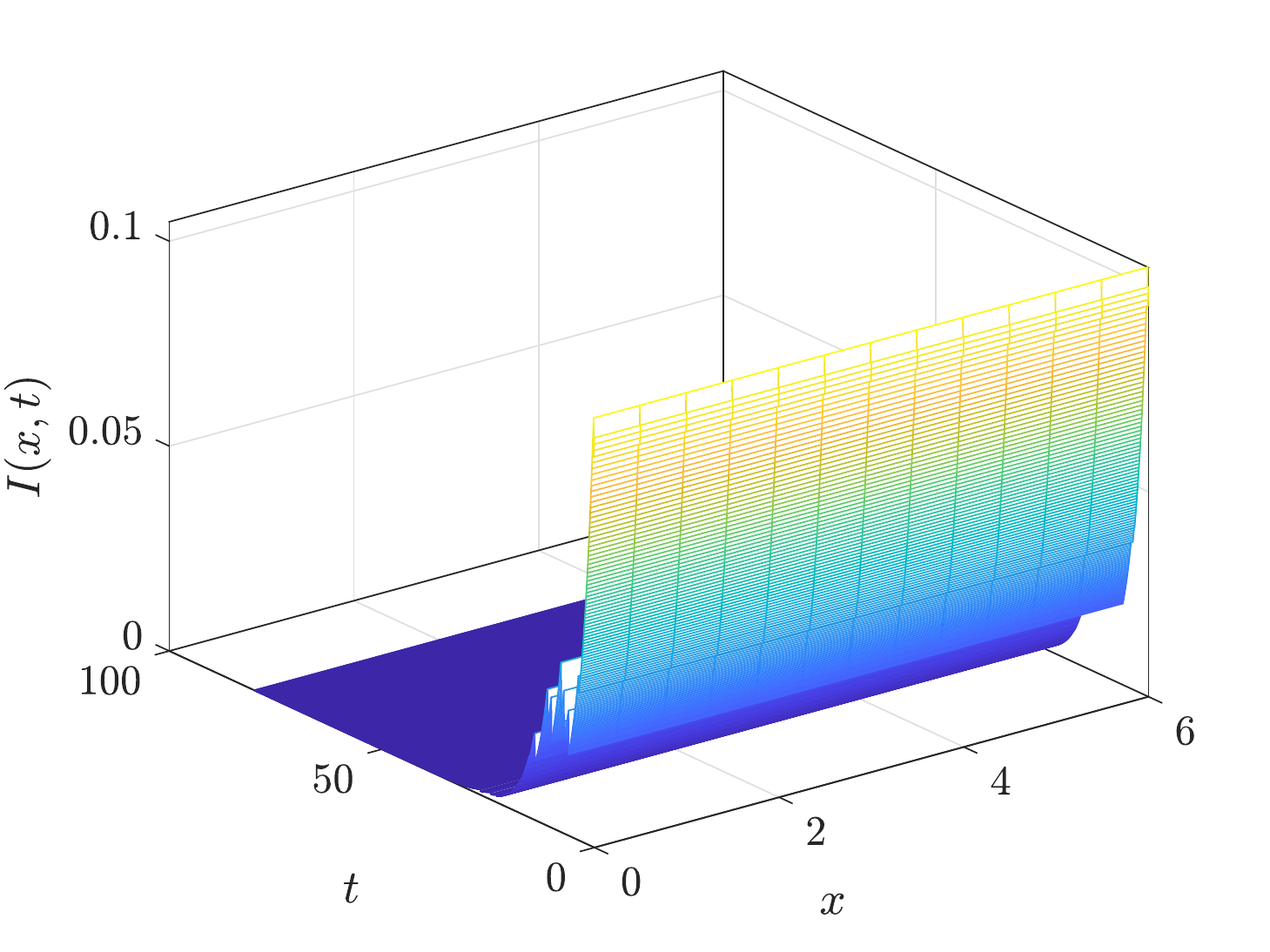}
			\end{minipage}
			\caption{Deterministic infected population (left), stochastic infected population with Gaussian noise (middle), stochastic infected population with pure-jump Lévy noise (right)}
			\label{figstandardi}
		\end{figure}
		\begin{figure}[H]
			\begin{minipage}{5cm}
				\includegraphics[height=4cm]{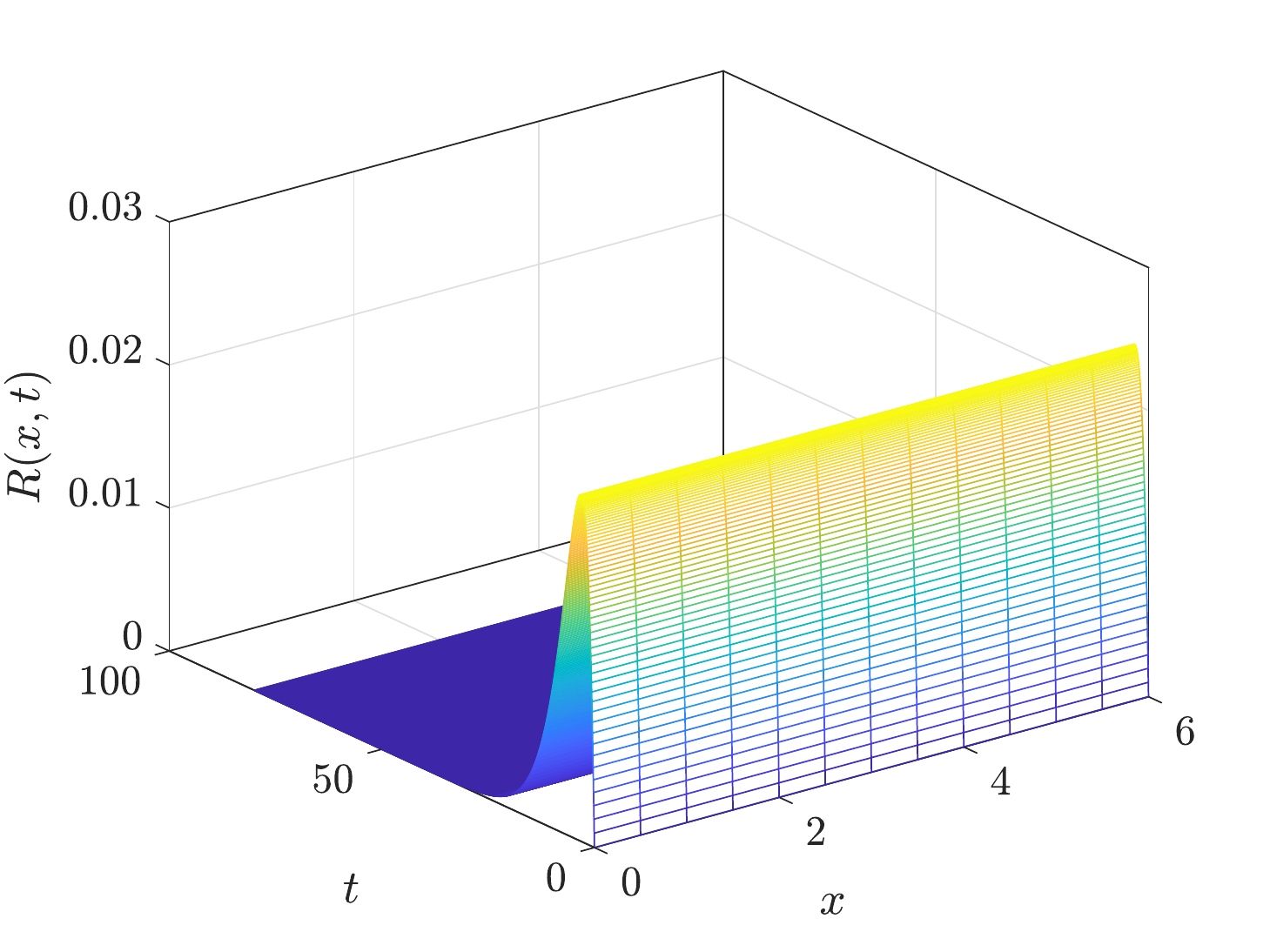}
			\end{minipage}
			\begin{minipage}{5cm}
				\includegraphics[height=4cm]{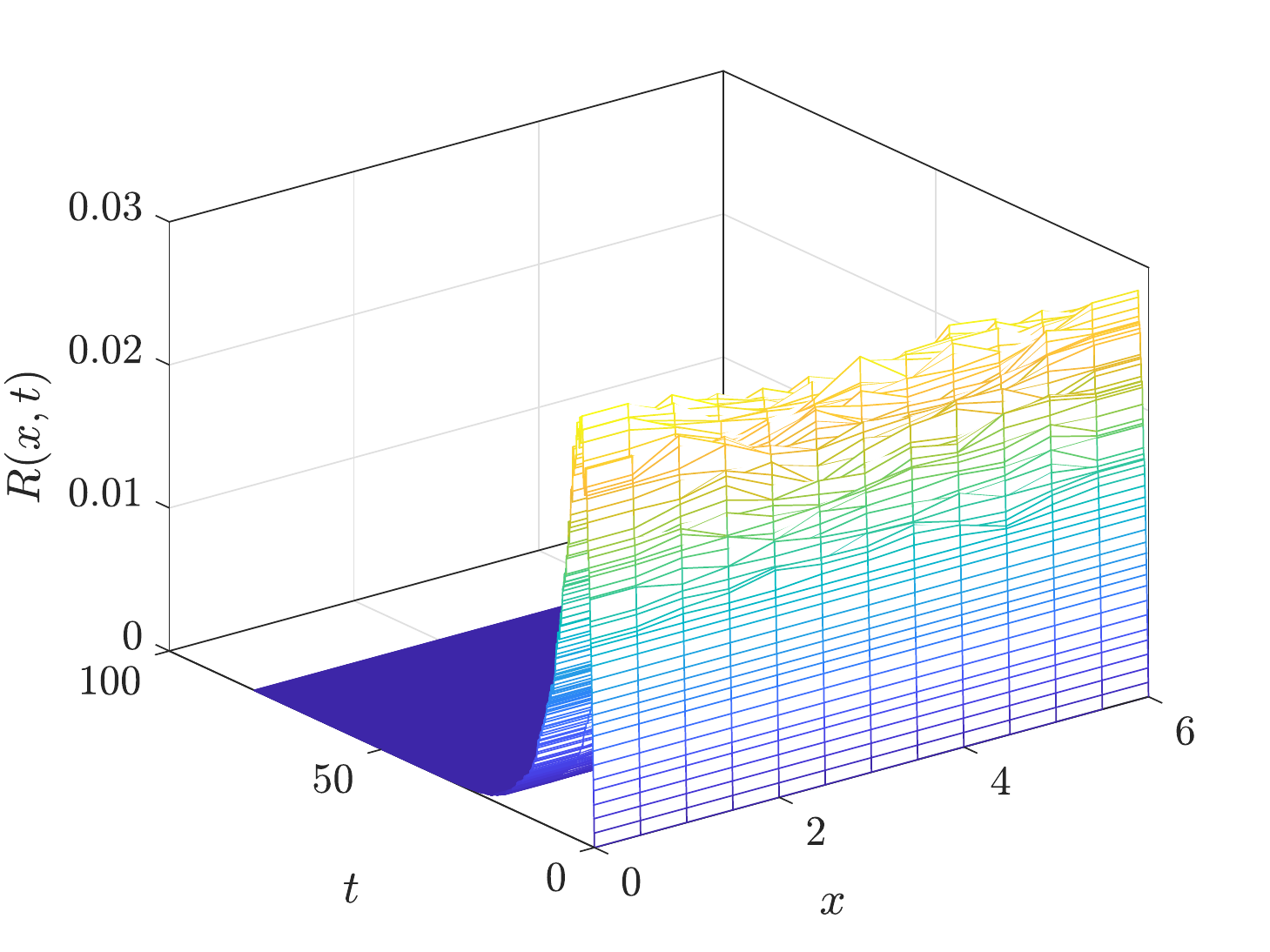}
			\end{minipage}
			\begin{minipage}{5cm}
				\includegraphics[height=4cm]{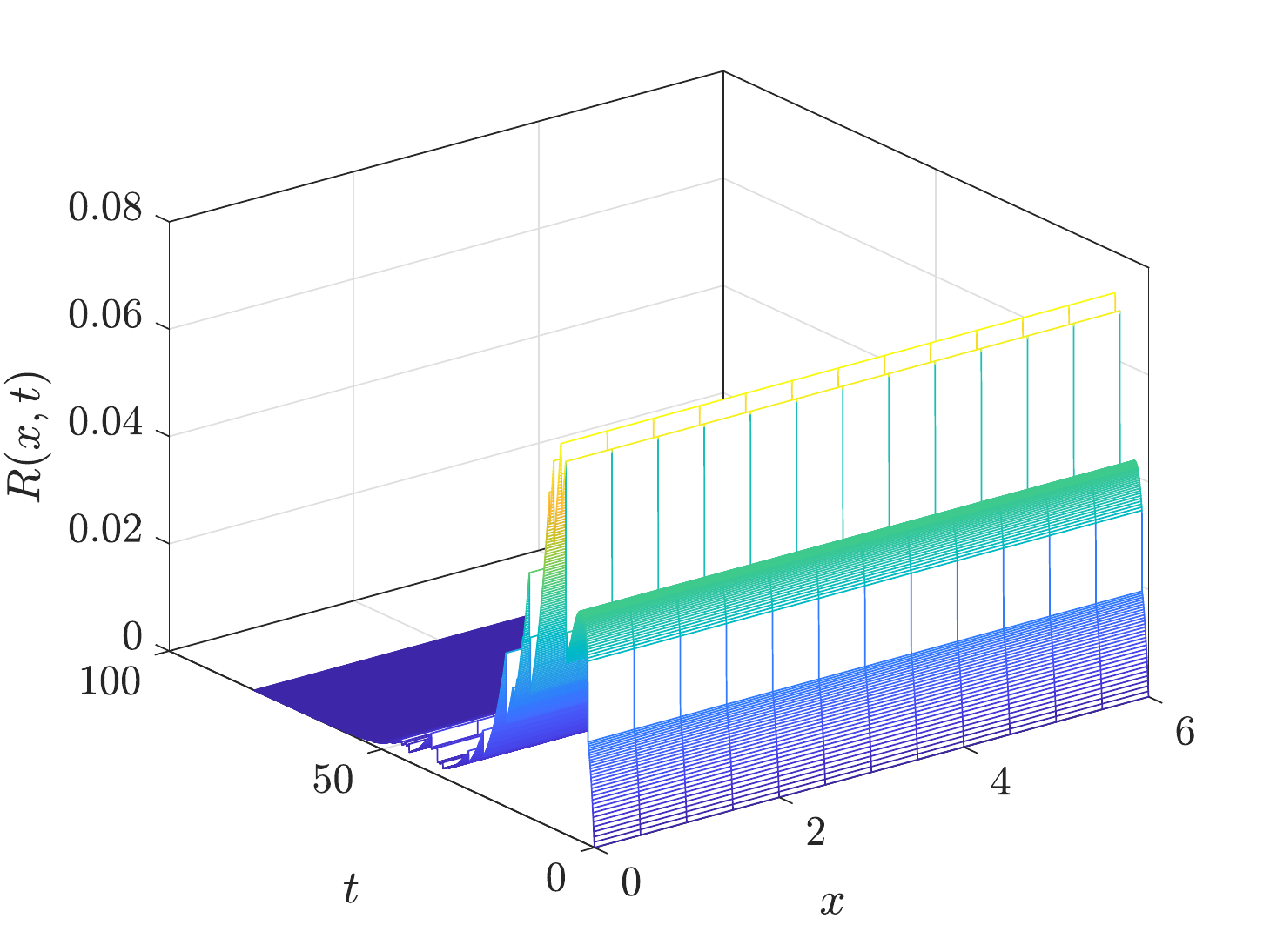}
			\end{minipage}
			\caption{Deterministic recovered population (left), stochastic recovered population with Gaussian noise (middle), stochastic recovered population with pure-jump Lévy noise (right)}
			\label{figstandardr}
		\end{figure}
		\begin{figure}[H]
			\begin{minipage}{5cm}
				\includegraphics[height=4.5cm]{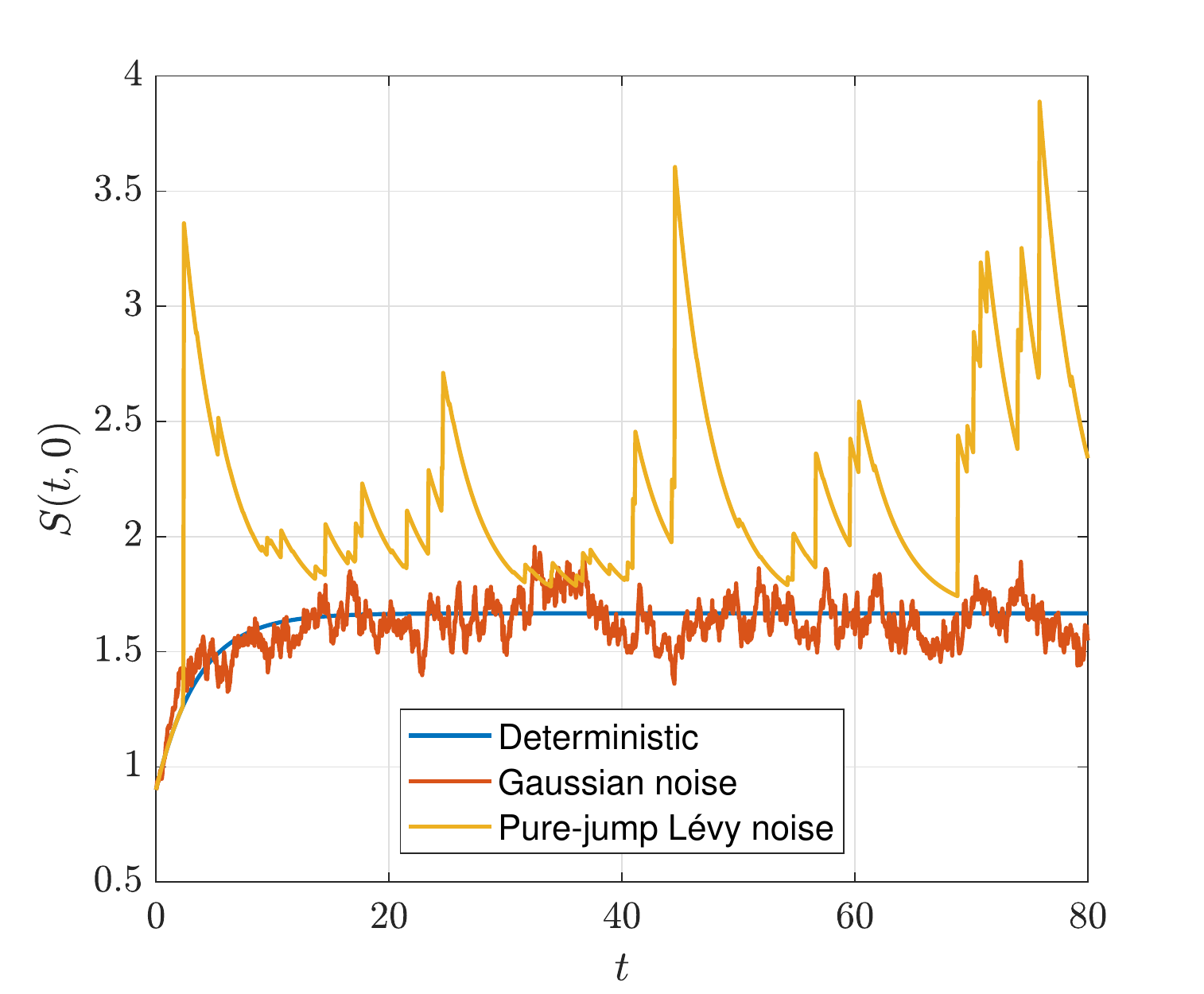}
			\end{minipage}
			\begin{minipage}{5cm}
				\includegraphics[height=4.5cm]{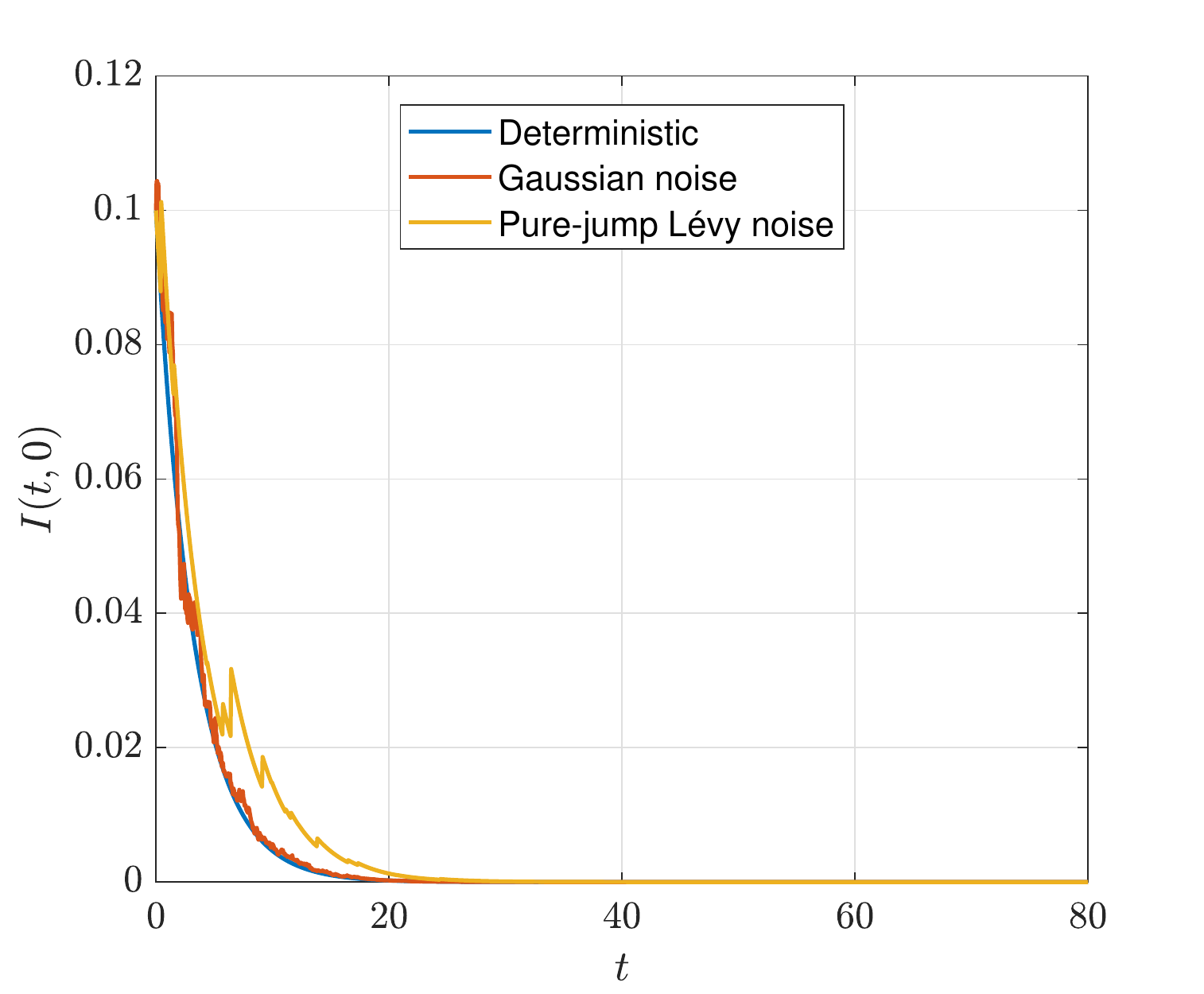}
			\end{minipage}
			\begin{minipage}{5cm}
				\includegraphics[height=4.5cm]{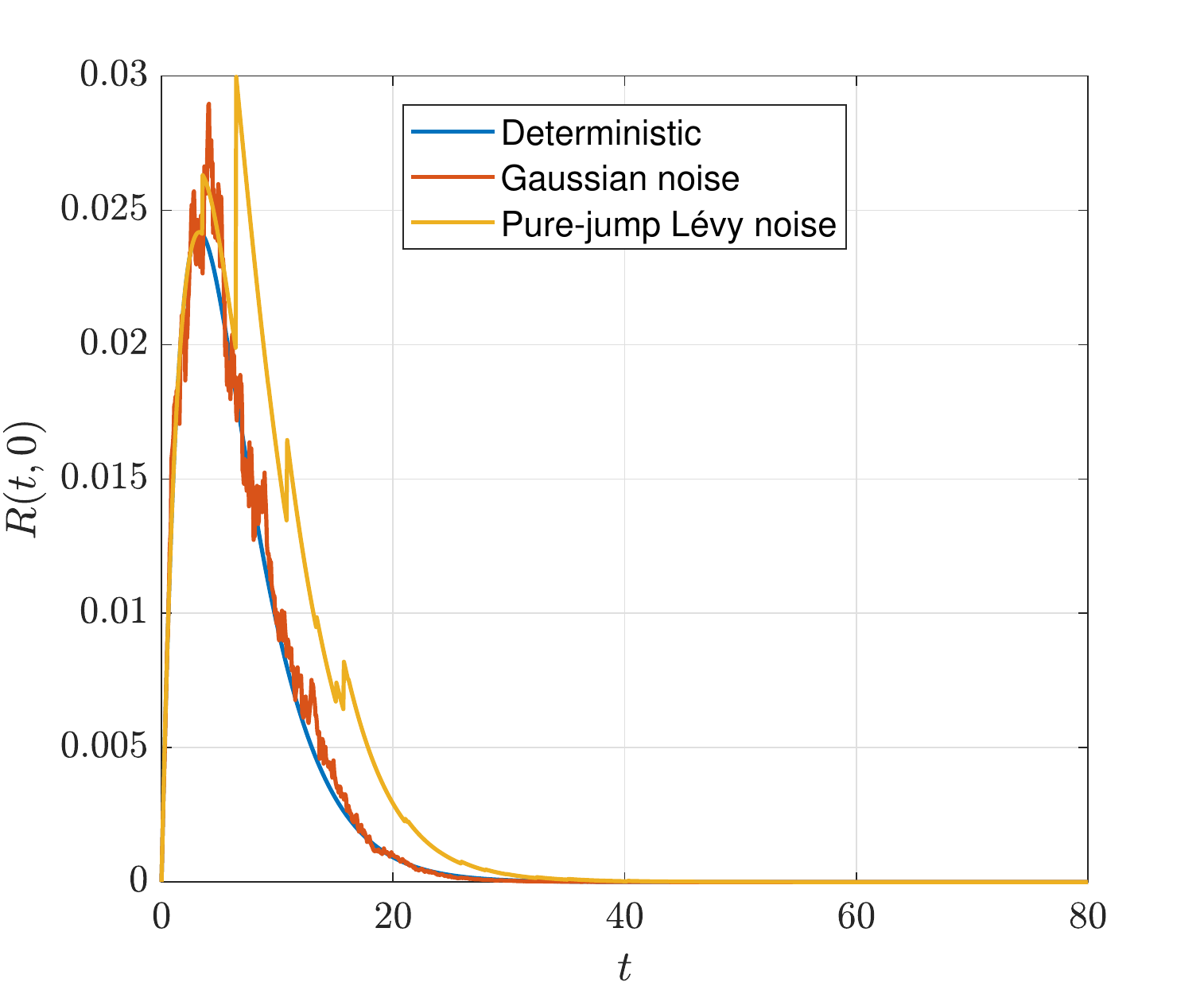}
			\end{minipage}
			\caption{Comparison of paths of the susceptible (left), infected (middle), recovered (right) at $x=0$.}
			\label{figstandardpaths}
		\end{figure}
		\subsection{Second example: Holling-type incidence rate}
		In this example, the chosen incidence function is given by \eqref{eq:saturated}, so that \ref{P3} is satisfied. The assigned values to the remaining input parameters in this case are given in Table \ref{tab:assigned2}. 
		\begin{table}[H]
			\centering
			\setlength\extrarowheight{2.5pt}
			\caption{Assigned values to the remaining parameters}
			\label{tab:assigned2}
			\begin{tabular}{|l|l|lll}
				\cline{1-2}
				Parameter&Assigned value  &  &  &  \\ \cline{1-2}
				$\Lambda$& $0.5$  &  &  &  \\ \cline{1-2}
				$\beta$& $0.4$ &  &  &  \\ \cline{1-2}
				$\mu $& $0.3$ & &  &  \\ \cline{1-2}
				$\gamma$& $0.05$ &  &  &  \\ \cline{1-2}
				$a$& $0.1$ &  &  &  \\ \cline{1-2}
			    $b$& $0$ &  &  &  \\ \cline{1-2}
				$\mathcal{C}_1, \mathcal{C}_2, \mathcal{C}_3$& $0.1$ &  &  &  \\ \cline{1-2}
				$\sigma_1, \sigma_2, \sigma_3$& $1$ &  &  &  \\ \cline{1-2}
			\end{tabular}
		\end{table}
		The assigned values to the parameters given in Table \ref{tab:assigned2} are in the aim of illustrating the effect of pure-jump Lévy noise in the case: 
		$
		\mathcal{R}_0:=\dfrac{\Lambda \beta}{\mu \left(\mu+\gamma\right)} > 1.
		$
		In this case, it is known that the endemic equilibrium point is globally asymptotically stable, meaning that the infection will persist within the population \cite{luo2019global}.
	
		Figures \ref{fighollings}-\ref{fighollingr} show the obtained solution in the deterministic case, stochastic case with Gaussian noise and stochastic case with pure-jump Lévy noise. It can be seen, in this case, that the persistence of the disease is remarkably higher in comparison to the deterministic case and the stochastic case with Gaussian noise. Additionally, similarly to the first case, from Fig. \ref{figstandardpaths}, one can observe a discontinuity in the paths of the obtained solution, in the case of pure-jump Lévy noise.
		\begin{figure}[H]
			\begin{minipage}{5cm}
				\includegraphics[height=4cm]{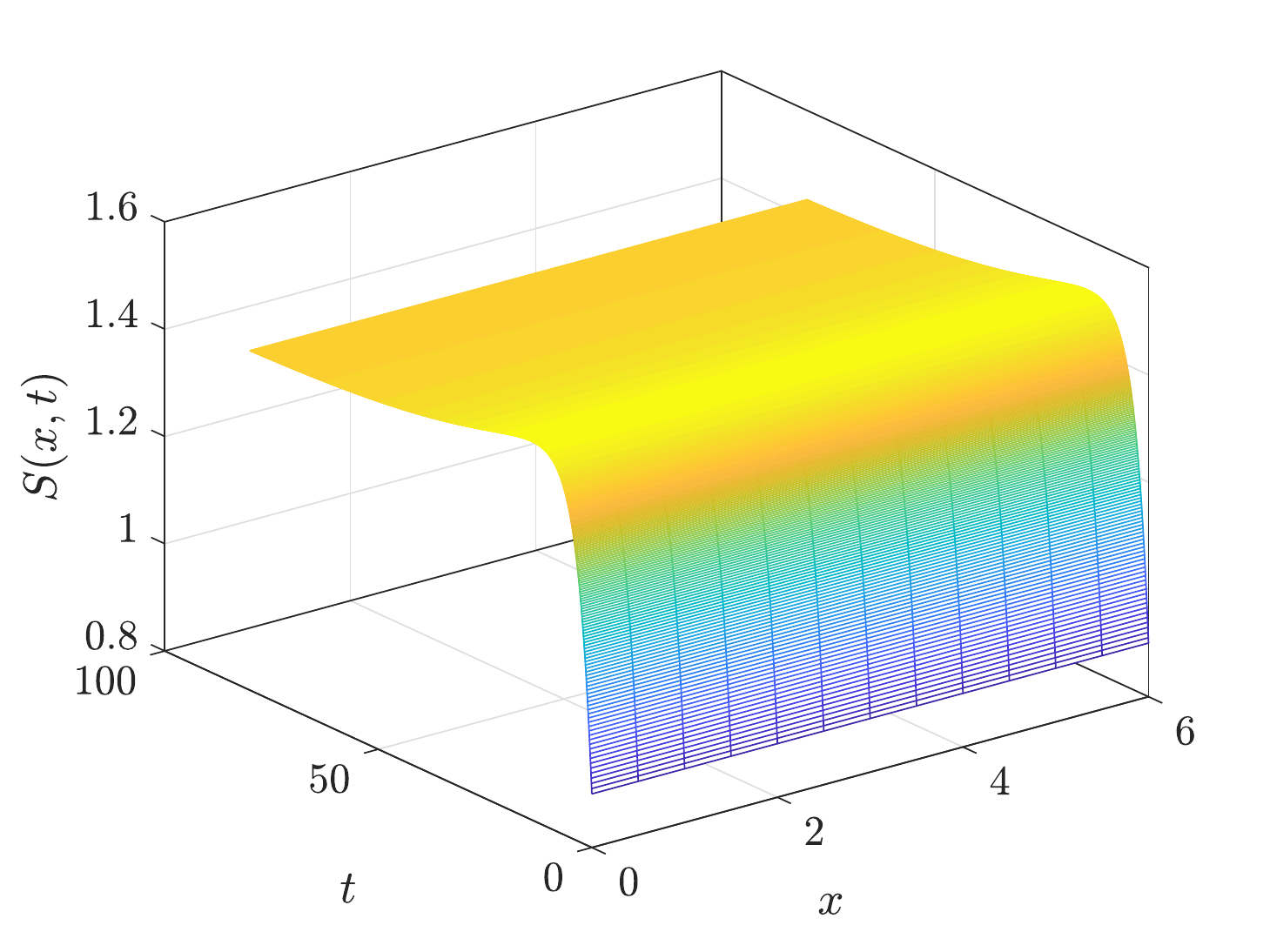}
			\end{minipage}
			\begin{minipage}{5cm}
				\includegraphics[height=4cm]{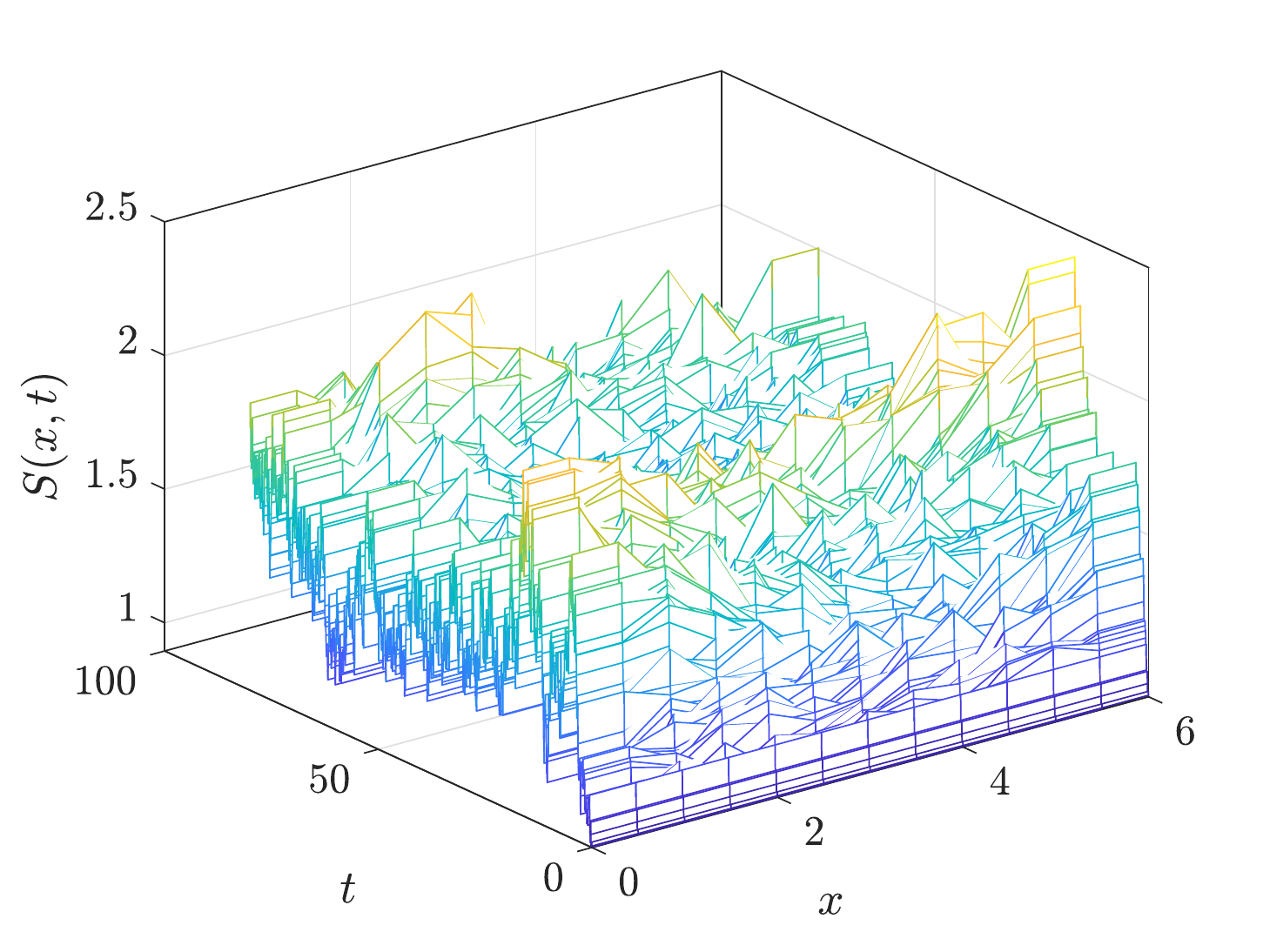}
			\end{minipage}
			\begin{minipage}{5cm}
				\includegraphics[height=4cm]{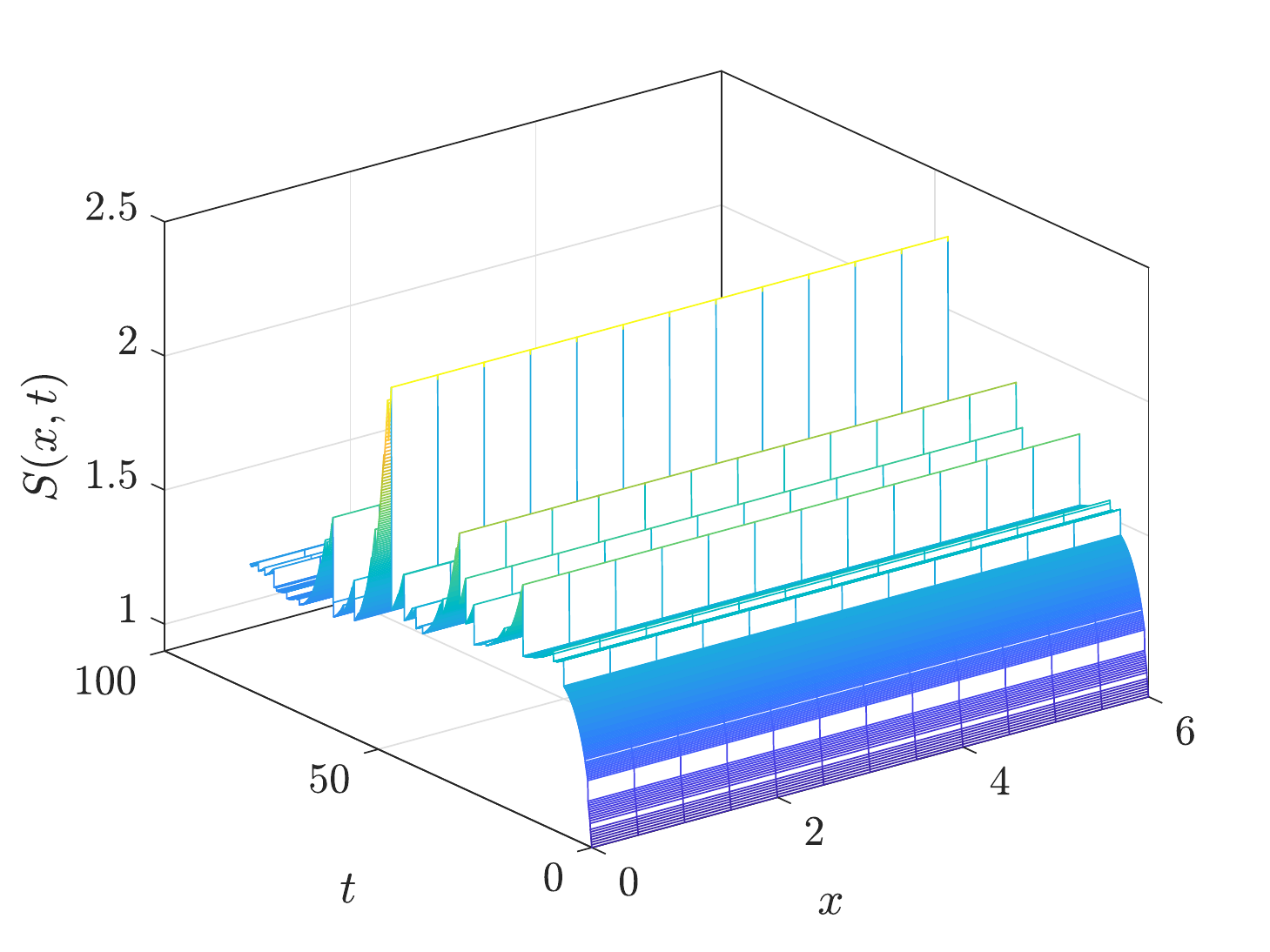}
			\end{minipage}
			\caption{Deterministic susceptible population (left), stochastic susceptible population with Gaussian noise (middle), stochastic susceptible population with pure-jump Lévy noise (right)}
			\label{fighollings}
		\end{figure}
		\begin{figure}[H]
			\begin{minipage}{5cm}
				\includegraphics[height=4cm]{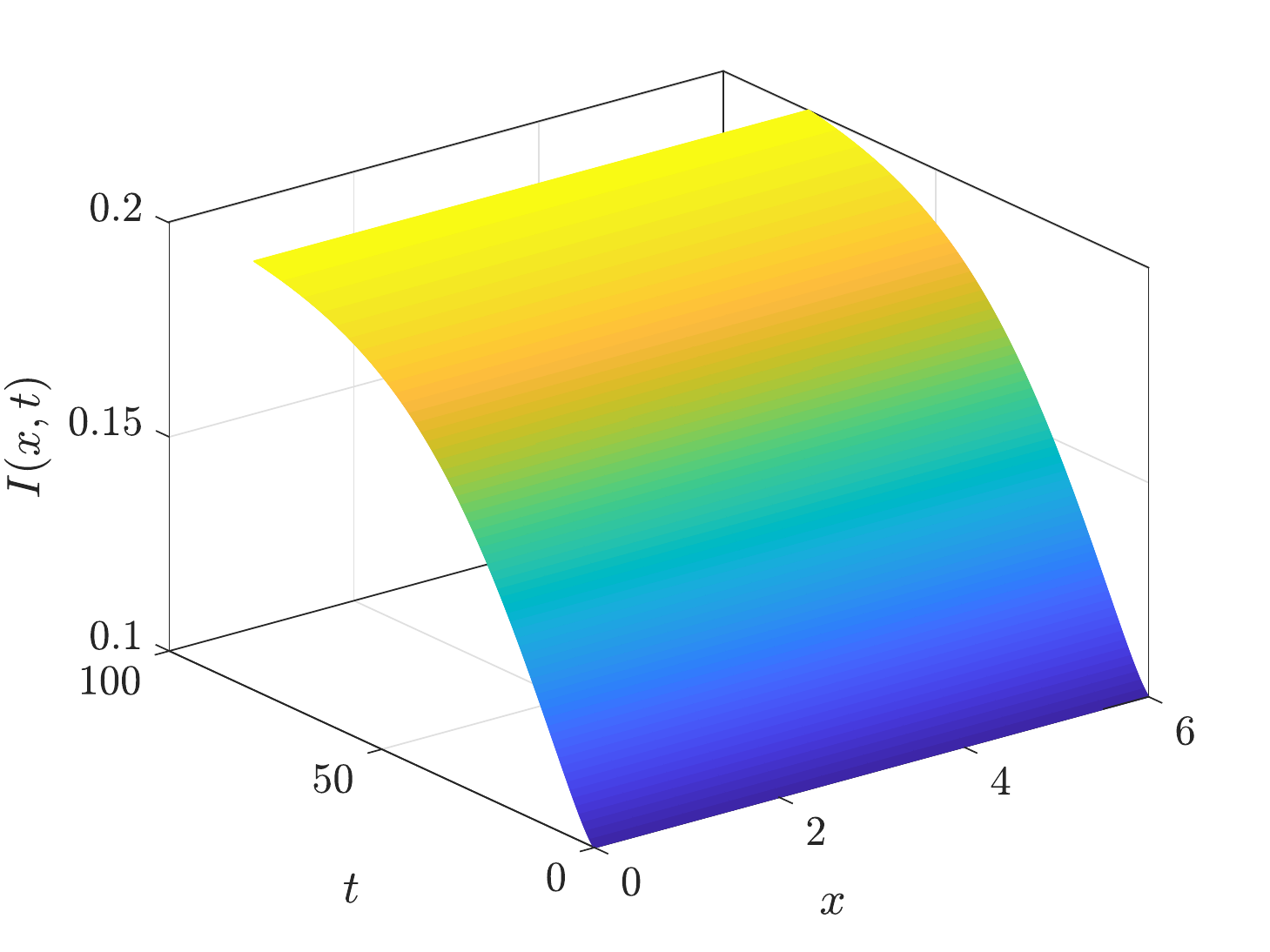}
			\end{minipage}
			\begin{minipage}{5cm}
				\includegraphics[height=4cm]{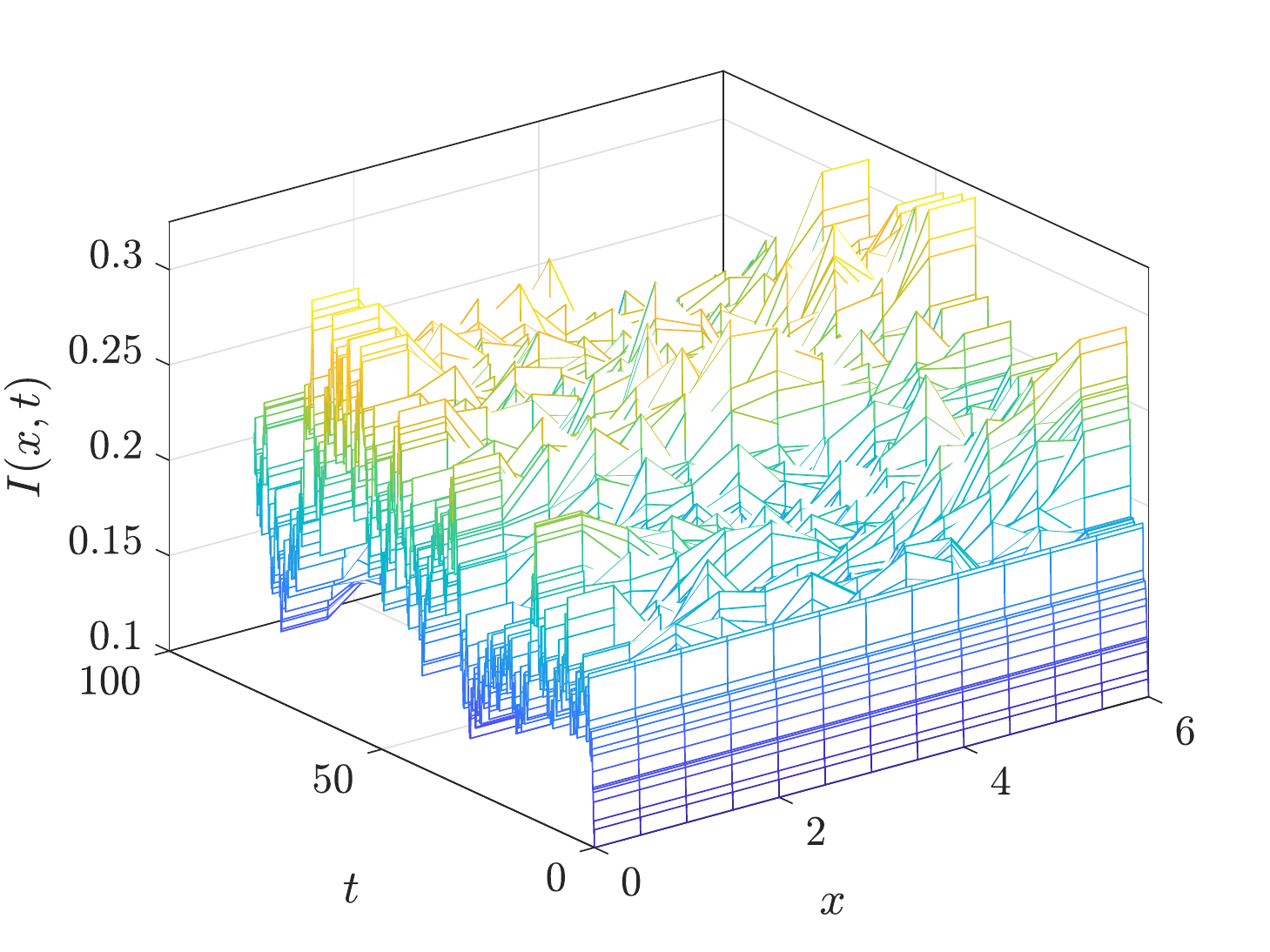}
			\end{minipage}
			\begin{minipage}{5cm}
				\includegraphics[height=4cm]{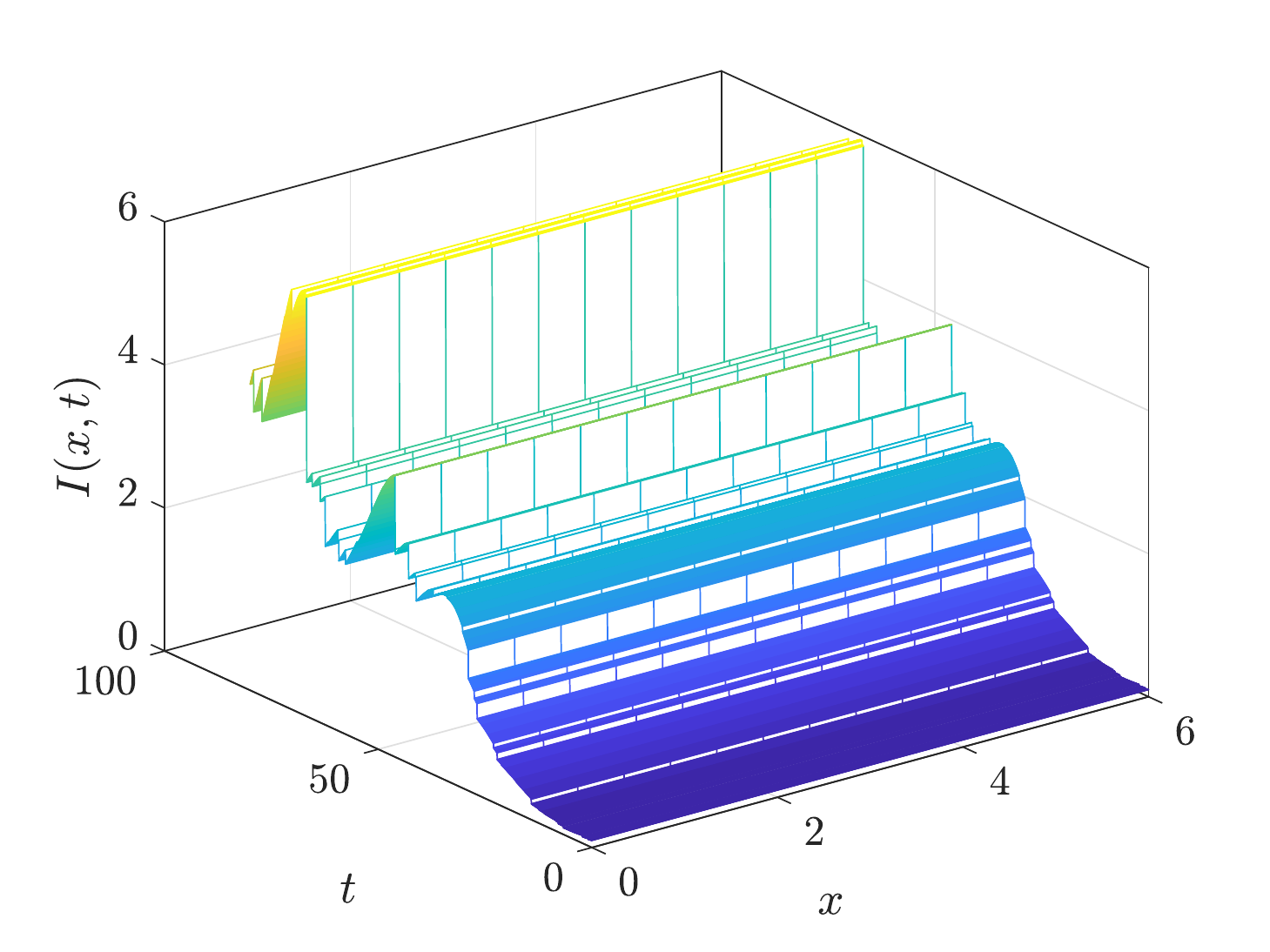}
			\end{minipage}
			\caption{Deterministic infected population (left), stochastic infected population with Gaussian noise (middle), stochastic infected population with pure-jump Lévy noise (right)}
			\label{fighollingi}
		\end{figure}
		\begin{figure}[H]
			\begin{minipage}{5cm}
				\includegraphics[height=4cm]{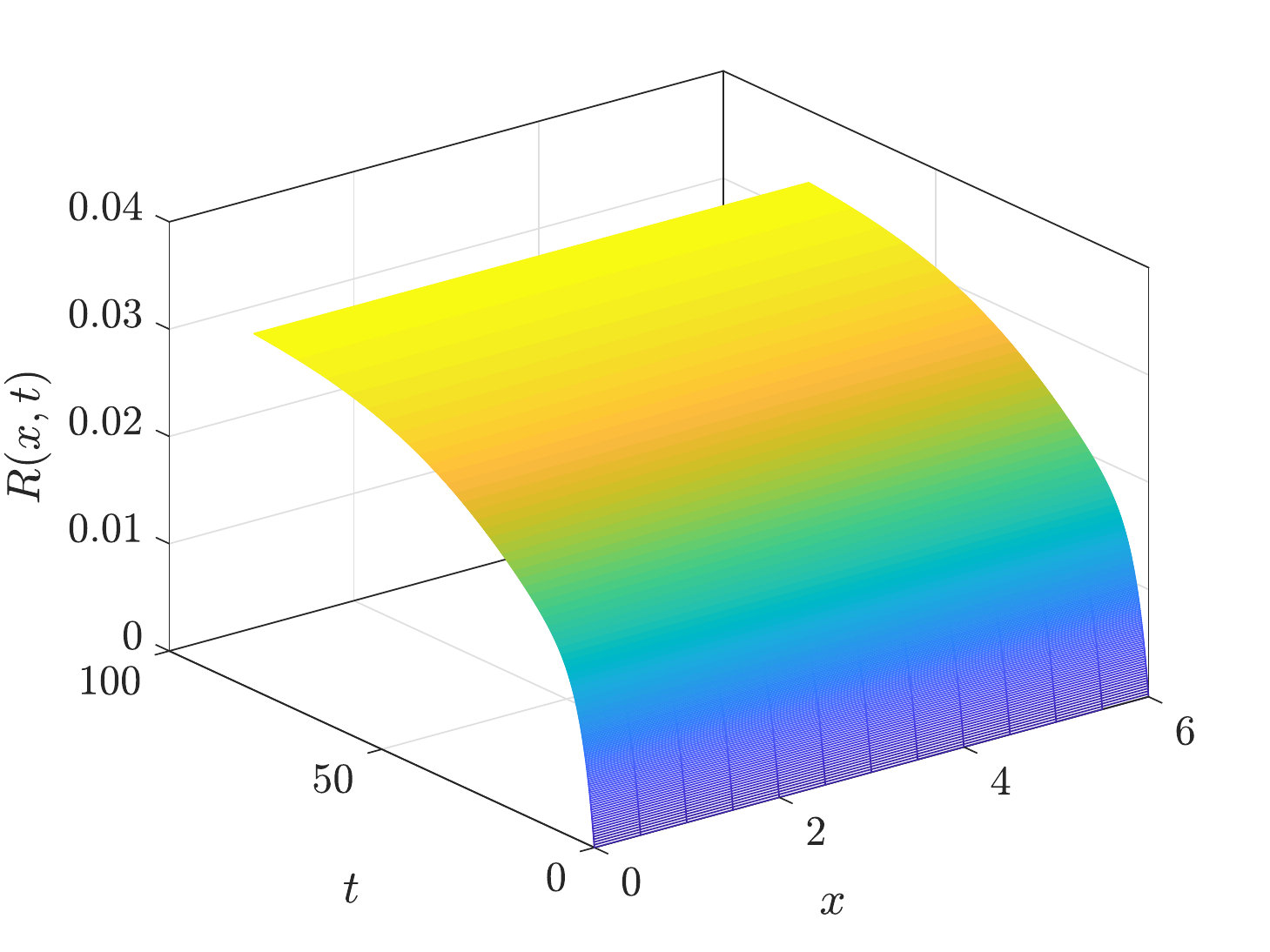}
			\end{minipage}
			\begin{minipage}{5cm}
				\includegraphics[height=4cm]{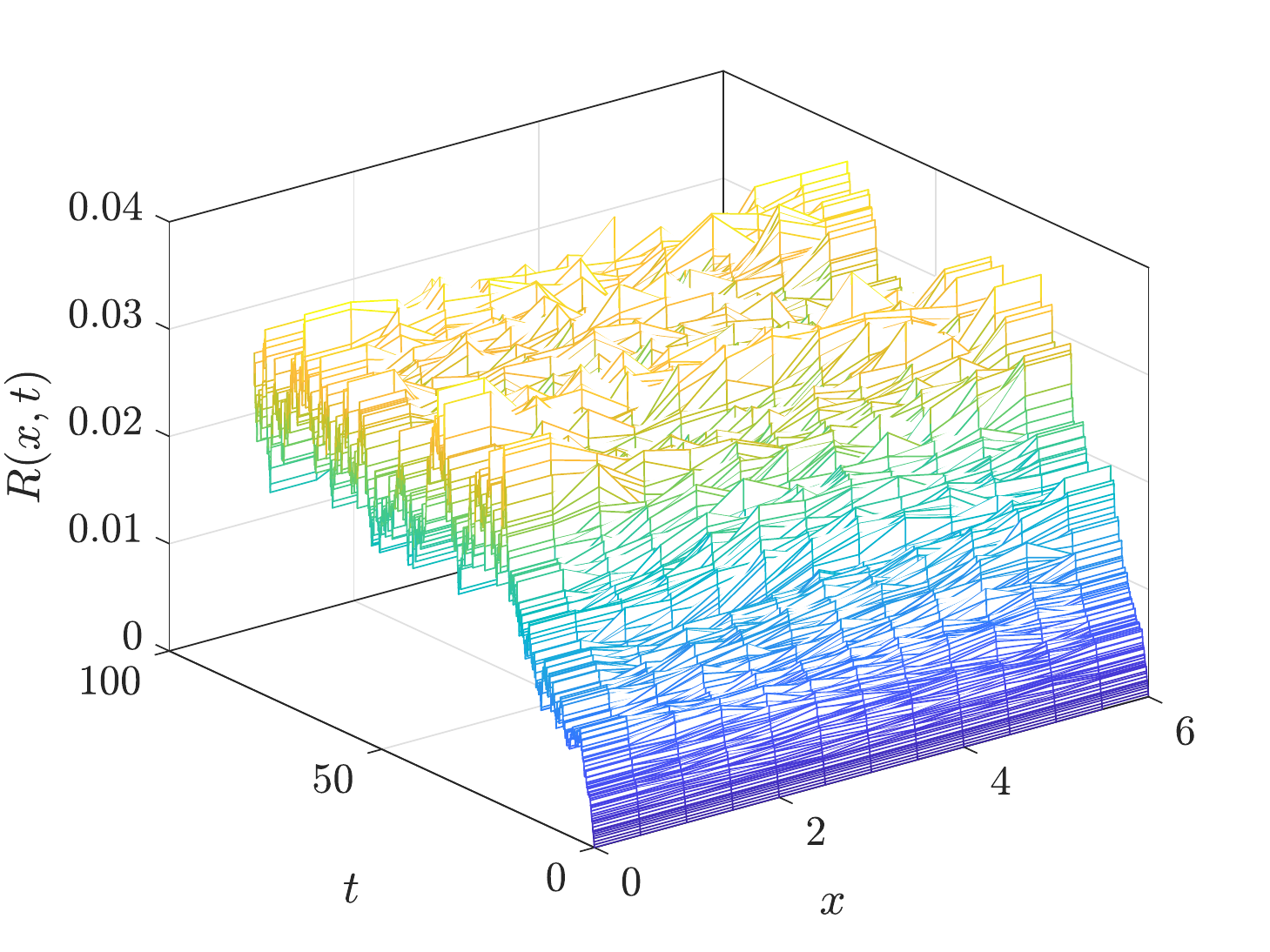}
			\end{minipage}
			\begin{minipage}{5cm}
				\includegraphics[height=4cm]{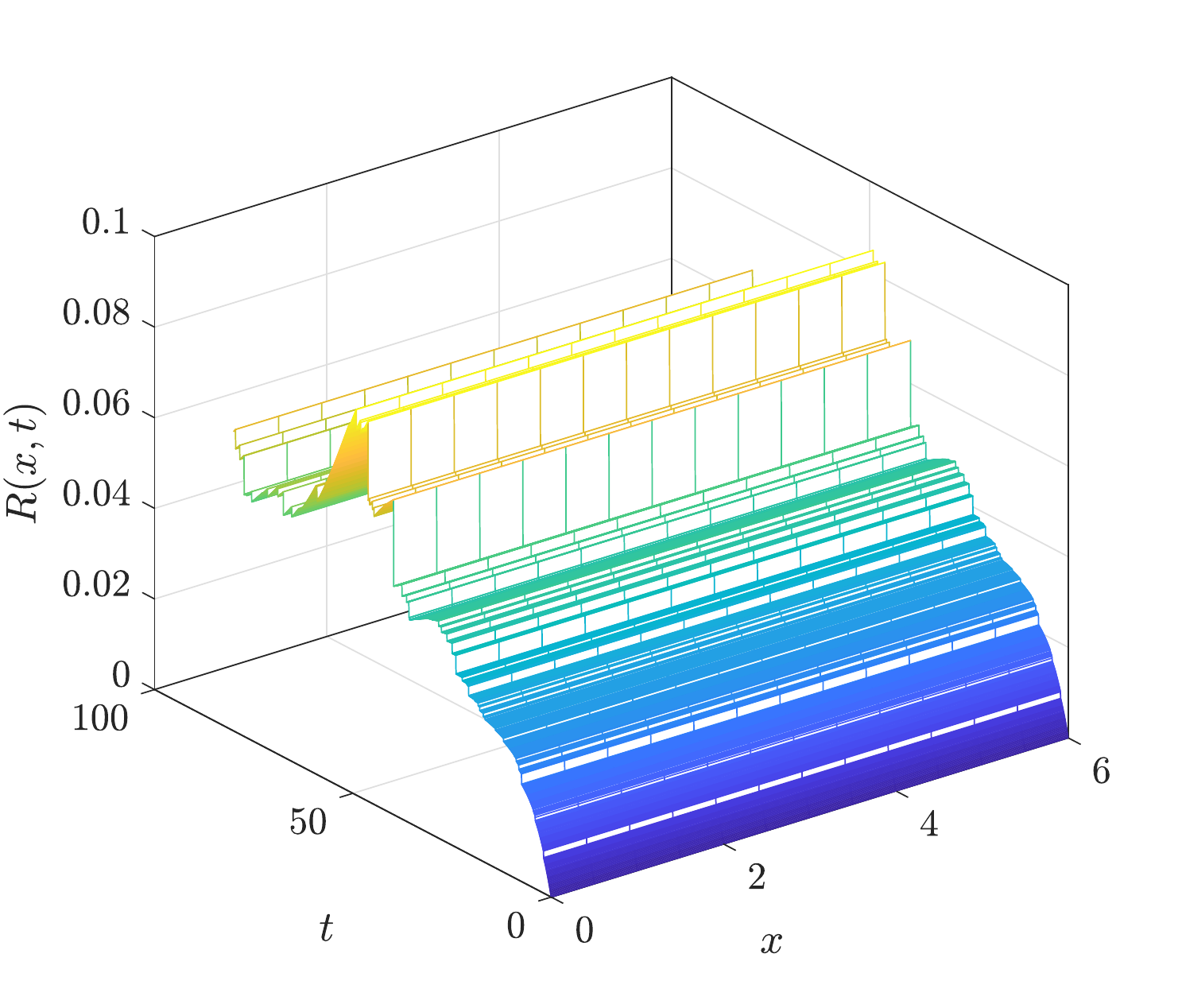}
			\end{minipage}
			\caption{Deterministic recovered population (left), stochastic recovered population with Gaussian noise (middle), stochastic recovered population with pure-jump Lévy noise (right)}
			\label{fighollingr}	
		\end{figure}
		\begin{figure}[H]
			\begin{minipage}{5cm}
				\includegraphics[height=4.5cm]{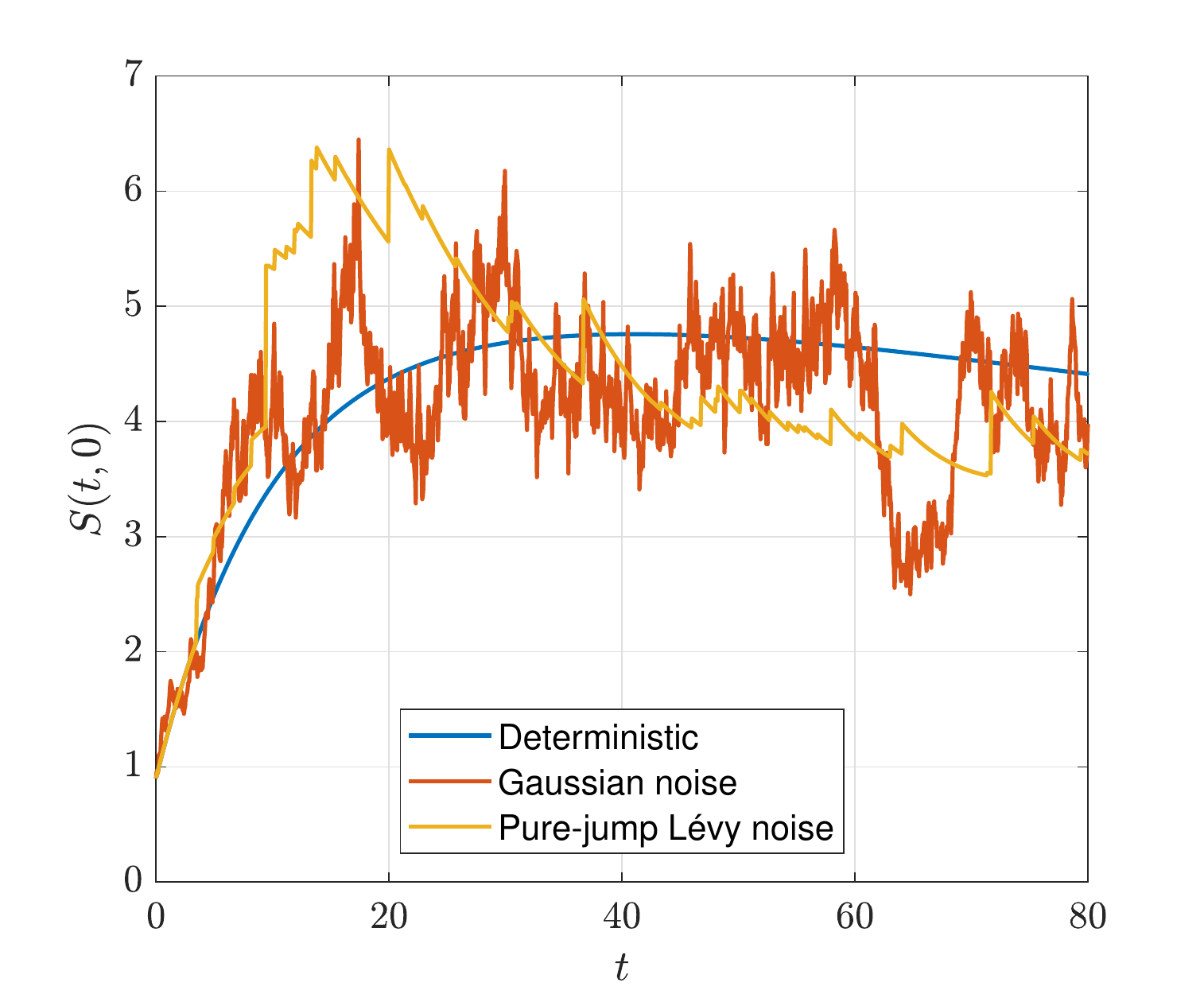}
			\end{minipage}
			\begin{minipage}{5cm}
				\includegraphics[height=4.5cm]{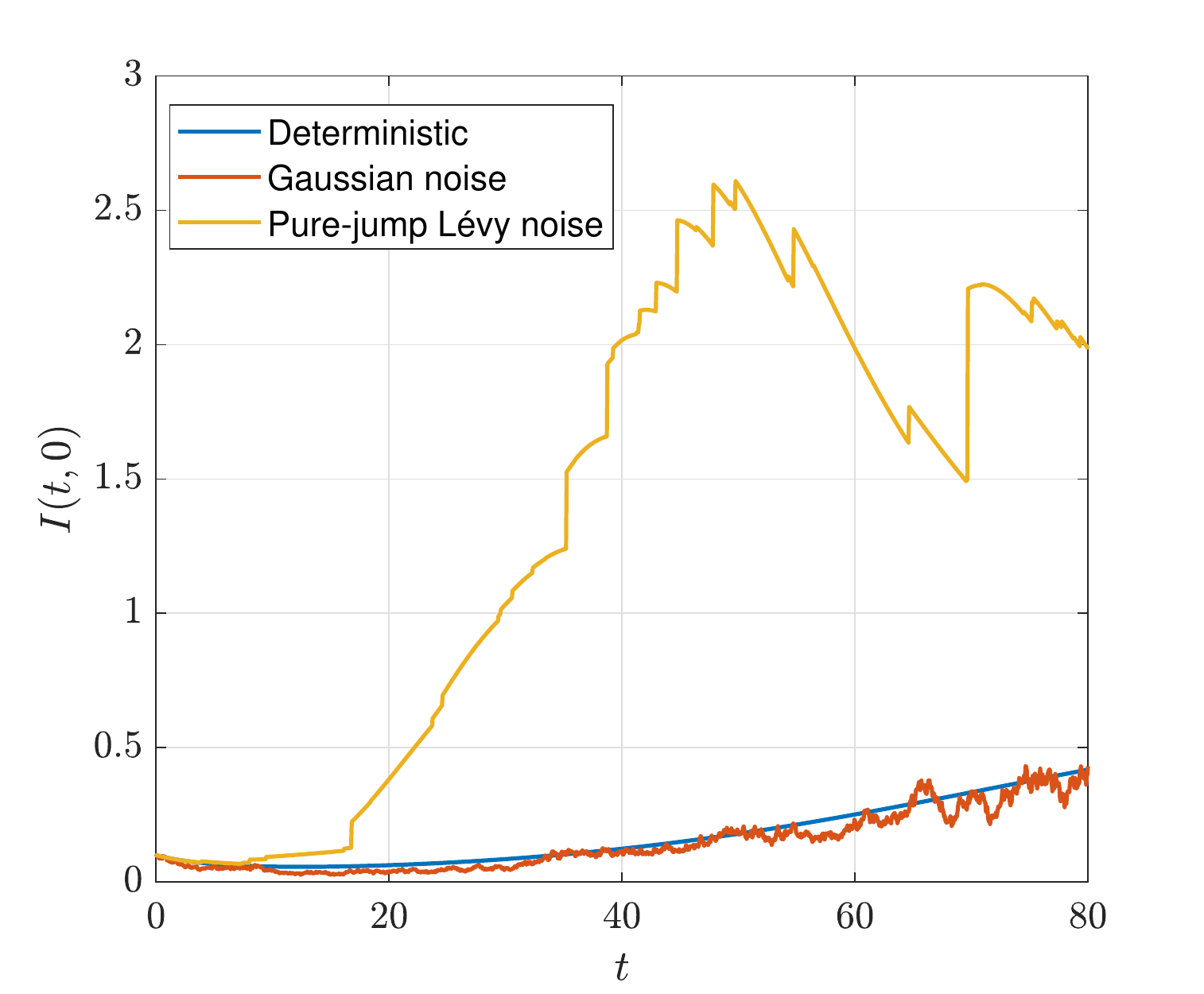}
			\end{minipage}
			\begin{minipage}{5cm}
				\includegraphics[height=4.5cm]{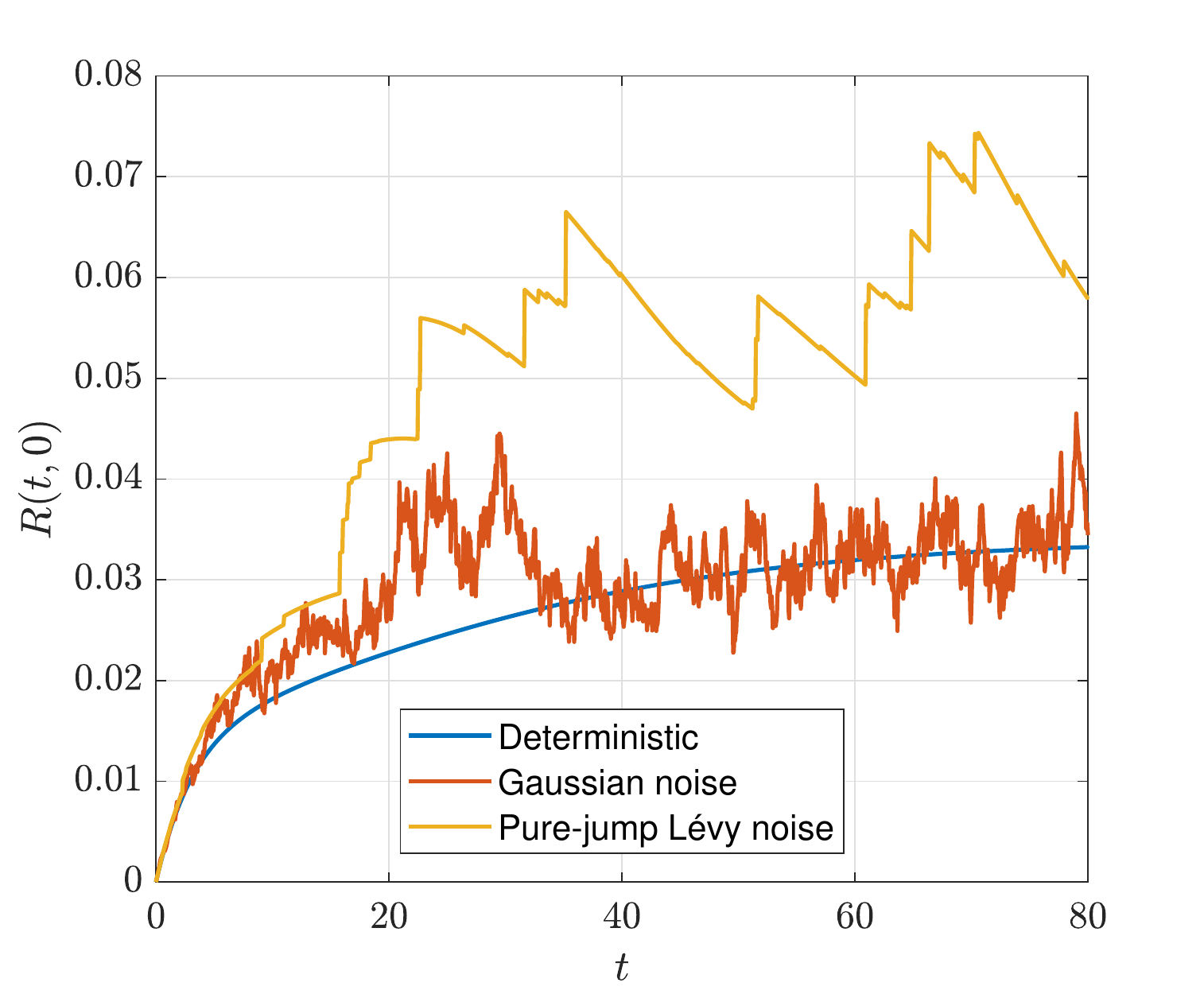}
			\end{minipage}
			\caption{Comparison of paths of the susceptible (left), infected (middle), recovered (right) at $x=0$.}
			\label{fighollingpaths}
		\end{figure}
		\section{Conclusion and open questions}\label{section6}
		In this work, we have proposed a new class of stochastic epidemic models which, as far as we know, is the first to incorporate the spatial dependency together with the massive discontinuous noise on the dynamics. Moreover, we have addressed its mathematical well-posedness and biological feasibility. It should be noted that the results established in this paper can be extended to other epidemic models falling into the same class and incorporating other biological characteristics, such as the temporary immunity, the incubation period, the vaccination procedure, the co-infection, etc; we mention for instance the models developed in \cite{li1995global,ng2003double,kyrychko2005global,liu2008svir,liu2013dynamics,wu2000homoclinic}. We should indicate that this was our first attempt to incorporate the discontinuous noise on the spatio-temporal dynamics of epidemics, and thus our focus was to mainly address the questions of mathematical well-posedness and biological feasibility which, now that they have been addressed, lay the ground-work for some further important emerging questions, which need to be tackled. These are briefly outlined as follows:
		\begin{itemize}
			\item \textbf{The case of the bilinear incidence function:} In this case, the Lipschitz property of the incidence function is merely local and the growth condition given in \ref{P3} is no longer verified. Hence, the truncation technique used in this paper only allows to prove the existence of local biologically feasible mild solutions. This result is not practical since, in general, there is no explicit information about the lifespan of the obtained solution. Hence, to address the global well-posedness, one needs to proceed with other techniques, such as variational methods (see e.g. \cite{mohan2022well}). 
			\item \textbf{The asymptotic behavior of the susceptible and infected populations:} Once the mathematical well-posedness and biological feasibility of a given epidemic model have been achieved, an interesting topic is deriving sufficient conditions guaranteeing the persistence and extinction of the susceptible and infected populations. Here, we mention that due to the result presented in Theorem \ref{exisstrongp2p3}, the difficulty arising in the lack of using Itô formula \cite[Theorem 3.5.3]{zhu2010study} for mild solutions has been addressed. Indeed, one can proceed to establish the aforementioned conditions for the sequence of strong solutions to Problem \eqref{approximativeprob}, then retrieve the desired result for mild solutions by a convergence argument. Thus, an interesting future direction would be to study the persistence and extinction for given explicit incidence functions satisfying \ref{P1} or \ref{P3}, such as the ones given by \eqref{eq:standard} and \eqref{eq:saturated}.   
			\item \textbf{The optimal control problem:} When the susceptible and infected populations are persistent, an emerging question is deriving the optimal control strategies steering the dynamics to a desired state, in which the densities of susceptible and infected populations are minimized. For instance, we refer to  \cite{mehdaoui2022optimal,mehdaouianalysis} where such results have been acquired in the deterministic case. However, as far as we know, the only existing result in the stochastic case driven by Gaussian noise can be found in the recent work of Shao \emph{et al.} \cite{shao2022necessary} and there are no results when it comes to stochastic spatio-temporal epidemic models driven by Lévy noise.
			\item \textbf{The parameter identification problem:} Once the mathematical well-posedness, biological feasibility and asymptotic analyses have been conducted, one focuses on the question of calibrating Model \eqref{eq:spdesirlevy}-\eqref{eq:init} using real data, which leads to what is known as the inverse problem. Namely, given observations of the susceptible, infected and recovered populations at time $T>0,$ one seeks the values assigned to the model parameters, such that the obtained solution provides the best approximation of the given observations. Such results have been derived in the case of deterministic spatio-temporal epidemic models by Xiang and Liu \cite{xiang2015solving} as well as Coronel \emph{et al.} \cite{coronel2021existence}. On the other hand, in the case of stochastic time-dependent epidemic models driven by Gaussian noise, this problem has been addressed by Mummert and Otunuga \cite{mummert2019parameter}. However, as far as we know, for stochastic spatio-temporal epidemic models driven by Gaussian or Lévy noise, this is still an open question. 
		\end{itemize}
		Due to the merit of the aforementioned questions, we will treat them independently in our next future works.
		  \section*{Acknowledgments}
		The author would like to express his sincere thanks to Professor Mouhcine Tilioua and  Professor Abdesslem Lamrani Alaoui  for their fruitful remarks and suggestions.
		\section*{Conflict of interest}
		No conflict of interest to be declared.
		\bibliographystyle{unsrt}
		\bibliography{MTL_SPAT_SIR_BIB.bib}
	\end{document}